\documentclass{amsart}

\title[Nonlinearity \&\ illfoundedness in hierarchy of consistency strength]{Nonlinearity and illfoundedness in the hierarchy of large cardinal consistency strength}

\author{Joel David Hamkins}
\address[Joel David Hamkins]
{O'Hara Professor of Philosophy and Mathematics, University of Notre Dame, 100 Malloy Hall, Notre Dame, IN 46556 USA}
\email{jdhamkins@nd.edu}
\urladdr{http://jdh.hamkins.org}

\thanks{Commentary can be made about this article on the author's blog at \href{http://jdh.hamkins.org/nonlinearity-in-the-hierarchy-of-consistency-strength}{http://jdh.hamkins.org/nonlinearity-in-the-hierarchy-of-consistency-strength}.}

\usepackage{latexsym,amsfonts,amsmath,amssymb,mathrsfs}
\usepackage{bbm}
\usepackage[hidelinks]{hyperref}
\usepackage{relsize}
\usepackage[rgb,dvipsnames]{xcolor}
\usepackage{tikz}
\usepackage{pgflibrarysnakes}
\usetikzlibrary{arrows,arrows.meta,petri,topaths,positioning,shapes,shapes.misc,patterns,calc,decorations.pathreplacing,hobby,snakes}
\usepackage{wrapfig} 
\usepackage{float}
\usepackage[utf8]{inputenc}
\RequirePackage{doi}

\usepackage{enumitem}

\newtheorem{theorem}{Theorem}
\newtheorem*{theorem*}{Theorem}

\newtheorem*{maintheorem*}{Main Theorem}
\newtheorem*{maintheorems*}{Main Theorems}
\newtheorem{corollary}[theorem]{Corollary}
\newtheorem*{corollary*}{Corollary}
\newtheorem*{corollaries*}{Corollaries}

\newtheorem{lemma}[theorem]{Lemma}

\newtheorem{observation}[theorem]{Observation}

\theoremstyle{definition}
\newtheorem{definition}[theorem]{Definition}
\newtheorem*{definition*}{Definition}

\newtheorem{question}[theorem]{Question}
\newtheorem*{question*}{Question}

\newtheorem*{questions*}{Questions}
\newtheorem*{mainquestion*}{Main Question} 
\newtheorem*{openquestion*}{Open Question} 
\theoremstyle{remark}

\newcommand{\QED}{\end{proof}}

\def\proclaim[#1]{{\bf #1}}
\def\BF#1.{{\bf #1.}}

\def\says#1:#2\par{\item[#1] #2\par}

\newcommand{\Dzamonja}{D\v{z}amonja}

\newcommand{\Godel}{G\"odel}

\newcommand{\N}{{\mathbb N}}
\renewcommand{\P}{{\mathbb P}}
\newcommand{\Q}{{\mathbb Q}}

\makeatletter
\newcommand{\dotminus}{\mathbin{\text{\@dotminus}}}
\newcommand{\@dotminus}{%
  \ooalign{\hidewidth\raise1ex\hbox{.}\hidewidth\cr$\m@th-$\cr}%
}
\makeatother

\newcommand{\from}{\mathbin{\vbox{\baselineskip=2pt\lineskiplimit=0pt
                         \hbox{.}\hbox{.}\hbox{.}}}}

\newcommand{\of}{\subseteq}

\newcommand{\set}[1]{\{\,{#1}\,\}}

\newcommand{\dom}{\mathop{\rm dom}}

\newcommand{\Con}{\mathop{{\rm Con}}}

\newcommand{\satisfies}{\models}

\newcommand{\proves}{\vdash}

\renewcommand{\setminus}{\raise.3ex\hbox{\rotatebox{-20}{$-$}}} 

\newcommand{\smalllt}{\mathrel{\mathchoice{\raise2pt\hbox{$\scriptstyle<$}}{\raise1pt\hbox{$\scriptstyle<$}}{\raise0pt\hbox{$\scriptscriptstyle<$}}{\scriptscriptstyle<}}}
\newcommand{\smallleq}{\mathrel{\mathchoice{\raise2pt\hbox{$\scriptstyle\leq$}}{\raise1pt\hbox{$\scriptstyle\leq$}}{\raise1pt\hbox{$\scriptscriptstyle\leq$}}{\scriptscriptstyle\leq}}}

   \def\DHLhksqrt#1#2{%
   \setbox0=\hbox{$#1\sqrt{#2\,}$}\dimen0=\ht0
   \advance\dimen0-0.2\ht0
   \setbox2=\hbox{\vrule height\ht0 depth -\dimen0}%
   {\box0\lower0.4pt\box2}}

\def\[#1]{\mathopen{\lbrack\!\lbrack}#1\mathclose{\rbrack\!\rbrack}}
\newbox\gnBoxA
\newbox\gnBoxB
\newdimen\gnCornerHgt
\setbox\gnBoxA=\hbox{\tiny$\ulcorner$}
\global\gnCornerHgt=\ht\gnBoxA
\newdimen\gnArgHgt
\def\gcode #1{%
\setbox\gnBoxA=\hbox{$#1$}%
\setbox\gnBoxB=\hbox{$\bar #1$}%
\gnArgHgt=\ht\gnBoxB%
\ifnum     \gnArgHgt<\gnCornerHgt \gnArgHgt=0pt%
\else \advance \gnArgHgt by -\gnCornerHgt%
\fi \raise\gnArgHgt\hbox{\tiny$\ulcorner$} \box\gnBoxA %
\raise\gnArgHgt\hbox{\tiny$\urcorner$}}
\newcommand{\UnderTilde}[1]{{\setbox1=\hbox{$#1$}\baselineskip=0pt\vtop{\hbox{$#1$}\hbox to\wd1{\hfil$\sim$\hfil}}}{}}
\newcommand{\Undertilde}[1]{{\setbox1=\hbox{$#1$}\baselineskip=0pt\vtop{\hbox{$#1$}\hbox to\wd1{\hfil$\scriptstyle\sim$\hfil}}}{}}
\newcommand{\undertilde}[1]{{\setbox1=\hbox{$#1$}\baselineskip=0pt\vtop{\hbox{$#1$}\hbox to\wd1{\hfil$\scriptscriptstyle\sim$\hfil}}}{}}
\newcommand{\UnderdTilde}[1]{{\setbox1=\hbox{$#1$}\baselineskip=0pt\vtop{\hbox{$#1$}\hbox to\wd1{\hfil$\approx$\hfil}}}{}}
\newcommand{\Underdtilde}[1]{{\setbox1=\hbox{$#1$}\baselineskip=0pt\vtop{\hbox{$#1$}\hbox to\wd1{\hfil\scriptsize$\approx$\hfil}}}{}}

\renewcommand{\iff}{\mathrel{\leftrightarrow}}
\newcommand{\Iff}{\mathrel{\Longleftrightarrow}}

\def\<#1>{\left\langle#1\right\rangle}

\newcommand{\ZFC}{{\rm ZFC}}
\newcommand{\ZF}{{\rm ZF}}

\newcommand{\GBC}{{\rm GBC}}

\newcommand{\CH}{{\rm CH}}

\newcommand{\DC}{{\rm DC}}

\newcommand{\MA}{{\rm MA}}

\newcommand{\PFA}{{\rm PFA}}

\newcommand{\PA}{{\rm PA}}
\newcommand{\PRA}{{\rm PRA}}

\newcommand{\cell}[1]{\boxit{\hbox to 17pt{\strut\hfil$#1$\hfil}}}
\newcommand{\head}[2]{\lower2pt\vbox{\hbox{\strut\footnotesize\it\hskip3pt#2}\boxit{\cell#1}}}
\newcommand{\boxit}[1]{\setbox4=\hbox{\kern2pt#1\kern2pt}\hbox{\vrule\vbox{\hrule\kern2pt\box4\kern2pt\hrule}\vrule}}
\newcommand{\Col}[3]{\hbox{\vbox{\baselineskip=0pt\parskip=0pt\cell#1\cell#2\cell#3}}}
\newcommand{\tapenames}{\raise 5pt\vbox to .7in{\hbox to .8in{\it\hfill input: \strut}\vfill\hbox to
.8in{\it\hfill scratch: \strut}\vfill\hbox to .8in{\it\hfill output: \strut}}}
\newcommand{\Head}[4]{\lower2pt\vbox{\hbox to25pt{\strut\footnotesize\it\hfill#4\hfill}\boxit{\Col#1#2#3}}}
\newcommand{\Dots}{\raise 5pt\vbox to .7in{\hbox{\ $\cdots$\strut}\vfill\hbox{\ $\cdots$\strut}\vfill\hbox{\
$\cdots$\strut}}}
\usepackage[backend=bibtex,style=alphabetic,maxbibnames=15,maxcitenames=6,dateabbrev=false]{biblatex}
\renewcommand{\UrlFont}{} 
\renewbibmacro{in:}{\ifentrytype{article}{}{\printtext{\bibstring{in}\intitlepunct}}} 
\DeclareFieldFormat{url}{{\UrlFont\url{#1}}} 
\DeclareFieldFormat{urldate}{
  (version \thefield{urlday}\addspace%
  \mkbibmonth{\thefield{urlmonth}}\addspace%
  \thefield{urlyear}\isdot)}
\DeclareFieldFormat{eprint:arxiv}{
\ifhyperref
    {\href{http://arxiv.org/abs/#1}{%
    arXiv\addcolon\nolinkurl{#1}}\iffieldundef{eprintclass}{}{\UrlFont{\mkbibbrackets{\thefield{eprintclass}}}}}
    {arXiv\addcolon\nolinkurl{#1}\iffieldundef{eprintclass}{}{\UrlFont{\mkbibbrackets{\thefield{eprintclass}}}}}}
    \bibliography{HamkinsBiblio,MathBiblio,WebPosts,PhilBiblio}
\newcommand{\ZFCcautious}{\ZFC^\circ}
\newcommand{\leqCon}{\leq_{\textup{Con}}}

\newcommand{\equivCon}{\equiv_{\textup{Con}}}

\begin{document}

\begin{abstract}
Many set theorists point to the linearity phenomenon in the hierarchy of consistency strength, by which natural theories tend to be linearly ordered and indeed well ordered by consistency strength. Why should it be linear? In this paper I present counterexamples, natural instances of nonlinearity and illfoundedness in the hierarchy of large cardinal consistency strength, as natural or as nearly natural as I can make them. I present diverse cautious enumerations of \ZFC\ and large cardinal set theories, which exhibit incomparability and illfoundedness in consistency strength, and yet, I argue, are natural. I consider the philosophical role played by ``natural'' in the linearity phenomenon, arguing ultimately that we should abandon empty naturality talk and aim instead to make precise the mathematical and logical features we had found desirable.
\end{abstract}

\maketitle

It is a mystery often mentioned in the foundations of mathematics, a fundamental phenomenon to be explained, that our best and strongest mathematical theories seem to be linearly ordered and indeed well-ordered by consistency strength. Given any two of the familiar large cardinal hypotheses, for example, generally one of them will prove the consistency of the other.

Why should it be linear? Why should the large cardinal notions line up like this, when they often arise from completely different mathematical matters? Measurable cardinals arise from set-theoretic issues in measure theory; Ramsey cardinals generalize ideas in graph coloring combinatorics; compact cardinals arise with compactness properties of infinitary logic. Why should these disparate considerations lead to principles that are linearly related by direct implication and consistency strength?

The phenomenon is viewed by many in the philosophy of mathematics as significant in our quest for mathematical truth. In light of \Godel\ incompleteness, after all, we must eternally seek to strengthen even our best and strongest theories. Is the linear hierarchy of consistency strength directing us along the elusive path, the ``one road upward'' as John Steel \cite{Steel2013:Godels-program-stanford-slides} describes it, toward the final, ultimate mathematical truth? That is the tantalizing possibility.

Meanwhile, we do know as a purely formal matter that the hierarchy of consistency strength is not actually well-ordered---it is ill-founded, densely ordered, and nonlinear. The statements usually used to illustrate these features, however, are weird self-referential assertions constructed in the Gödelian manner via the fixed-point lemma---logic-game trickery, often dismissed as unnatural.

Many set theorists claim that amongst the \emph{natural} assertions, however, the consistency strengths remain linearly ordered and indeed well ordered. H. Friedman \cite{Friedman1998:Proof-theoretic-degrees} refers to ``the apparent comparability of naturally occurring logical strengths as one of the great mysteries of [the foundations of mathematics].''
\medskip\goodbreak

Andrés Caicedo says,
\begin{quotation}
It is a remarkable empirical phenomenon that we indeed have comparability for natural theories. We expect this to always be the case, and a significant amount of work in inner model theory is guided by this belief.
\cite{MO59800.Caicedo:Nonlinearity-of-consistency-strength}
\end{quotation}\goodbreak

Stephen G. Simpson writes:
\begin{quote}
It is striking that a great many foundational theories are linearly ordered by $<$. Of course it is possible to construct pairs of artificial theories which are incomparable under $<$. However, this is not the case for the ``natural'' or non-artificial theories which are usually regarded as significant in the foundations of mathematics. The problem of explaining this observed regularity is a challenge for future foundational research. \cite{Simpson2009:Godel-hierarchy-and-reverse-mathematics}
\end{quote}

John Steel writes ``The large cardinal hypotheses [the ones we know] are themselves wellordered by consistency strength,'' and he formulates what he calls the ``vague conjecture''\label{Steel-conjecture} asserting that
\begin{quote}
If $T$ is a natural extension of \ZFC, then there is an extension $H$ axiomatized by large cardinal hypotheses such that $T \equivCon H$. Moreover, $\leqCon$ is a prewellorder of the natural extensions of $\ZFC$. In particular, if $T$ and $U$ are natural extensions of \ZFC, then either $T \leqCon U$ or $U\leqCon T$. \cite{Steel2014:Godels-program}
\end{quote}

Peter Koellner writes
\begin{quote}
Remarkably, it turns out that when one restricts to those theories that ``arise in nature'' the interpretability ordering is quite simple: There are no descending chains and there are no incomparable elements—the interpretability ordering on theories that ``arise in nature'' is a wellordering. \cite{Koellner2011:SEP-independence-large-cardinals}
\end{quote}

Let me refer to this position as the \emph{natural linearity position}, the assertion that all natural assertions of mathematics are linearly ordered by consistency strength. The strong form of the position, asserted by some of those whom I have cited above, asserts that the natural assertions of mathematics are indeed well-ordered by consistency strength. By all accounts, this view appears to be widely held in large cardinal set theory and the philosophy of set theory.

Despite the popularity of this position, I should like in this article to set myself a hard task, to explore the contrary view and directly to challenge the natural linearity position.

\begin{mainquestion*}
Can we find natural instances of nonlinearity and illfoundedness in the hierarchy of consistency strength?
\end{mainquestion*}

It will be a hard task, but I shall try my best.

\section{Formal instances of consistency-strength nonlinearity}\label{Section.Formal-nonlinearity}

Let me begin by setting aside the naturality requirement (for this section only) and reviewing as a purely formal matter that both nonlinearity and ill-foundedness occur in the hierarchy of consistency strength. This will be established with certain self-referential sentences constructed via the \Godel\ fixed-point lemma---precisely the sentences often dismissed as unnatural. The results of this section are well known and have essentially become part of the mathematical logic folklore; I shall mention several references presently.

The consistency-strength order relation at the center of the discussion is defined as follows, where $\Con(T)$ is the assertion ``$T$ is consistent,'' expressed using a given arithmetically expressible definition of the theory $T$.

\begin{definition}\label{Definition.Consistency-hierarchy}\rm
Theory $S$ is below theory $T$ in consistency strength, written $S\leq T$, if the implication $\Con(T)\to\Con(S)$ is provable in a base theory,  fixed for the purpose of this relation. The theory $S$ is strictly below $T$, written $S<T$, if $S\leq T$ but not conversely; and the theories are equiconsistent, written $S\equiv T$, when $S\leq T$ and $T\leq S$, which provides a hierarchy of degrees of consistency strength.
\end{definition}

One commonly sees Peano arithmetic \PA\ as the base theory in arithmetic contexts and \ZFC\ in set theory. Much weaker base theories, however, such as primitive recursive arithmetic \PRA, actually suffice for nearly all of the usual features one seeks in the hierarchy of consistency strength, whether in arithmetic or set theory, including large cardinal set theory. A strong base theory, such as \ZFC, can erase distinctions in the hierarchy that a weaker theory reveals; the iterated consistency tower over \PA, for example, with $\Con(\PA)$, $\Con(\PA+\Con(\PA))$ and so on, is trivialized in \ZFC, which proves outright all those consistency statements. Meanwhile, instances of nonlinearity are more compelling over a stronger base theory than a weak one, since incomparability persists to weaker base theories; and so in this article I shall adopt \PA\ as the default base theory for arithmetic and \ZFC\ as the default base theory for set theory, although very little of my analysis depends on this.

Let me point out a subtle intentional nature of the hierarchy of consistency strength. Namely, the order $S\leq T$ is not defined on theories as sets of sentences, but is rather defined on theories as described by an arithmetic assertion. It is the description of the theory $T$, after all, rather than the set of sentences, that is used when formulating the assertion $\Con(T)$. Feferman \cite{Feferman1960:Arithmetization-of-metamathematics-in-a-general-setting} realized that a given theory $T$ can be described in different ways in such a way that the corresponding consistency assertions $\Con(T)$ are inequivalent and indeed of arbitrary consistency strength, and this phenomenon will recur in parts of the later analysis of this article. A similar phenomenon arises in proof theory with proof-theoretic ordinals, which are often based on the presentation of the ordinal rather than the ordinal itself; see \cite{Walsh2021:On-the-hierarchy-of-natural-theories}.

A convenient sufficient condition for the strict relation $S<T$, for a consistent theory $T$ extending the base theory, occurs when $T\proves\Con(S)$. In this case $\Con(T)$ implies \hbox{$\Con(T+\Con(S))$}, which implies $\Con(S)$, but $\Con(S)$ cannot prove $\Con(T)$ over the base theory, for then the base theory plus $\Con(S)$ would prove its own consistency, contrary to the incompleteness theorem. This method enables one to prove instances of the strict relation $S<T$ without ever having explicitly to prove an instance of nonprovability, that is, $T\nleq S$, since this part is in effect offloaded to the incompleteness theorem. In the large cardinal hierarchy, many instances of the strict relation in consistency strength are often proved in just this way---if there is a Mahlo cardinal $\kappa$, for example, then $V_\kappa$ is a set-sized model of \ZFC\ with a proper class of inaccessible cardinals, showing the consistency of this theory, which is therefore strictly weaker than a Mahlo cardinal.

In the research literature, not everyone is using the same hierarchy of strength; there are several closely related notions. First, of course, there is the naive hierarchy of provable direct implication, but nearly everyone recognizes that this isn't really what we want, and we have numerous natural instances of nonlinearity for this. For example, $\ZFC+\CH$ and $\ZFC+\neg\CH$ are clearly incomparable by provable implication over $\ZFC$, and also $\ZFC$ plus a proper class of inaccessible cardinals is incomparable by direct implication with $\ZFC$ plus a Mahlo cardinal, and there are many further examples like this.

Per Lindstrom \cite{Lindstrom2003:Aspects-of-incompleteness} analyzes the hierarchy of interpretability strength, which is closely related to but not identical with the hierarchy of consistency strength; nevertheless, much of Lindstrom's analysis carries over easily to consistency strength. Koellner \cite{Koellner2011:SEP-independence-large-cardinals} similarly treats the interpretability hierarchy in his discussion of large cardinal strength. It may be that for most purposes including philosophical analysis we should ultimately be using interpretability strength rather than consistency strength. These two hierarchies are different, however, even in the case of \ZFC\ and \Godel-Bernays set theory \GBC, since these two theories are equiconsistent but not mutually interpretable.

Harvey Friedman \cite{Friedman1998:Proof-theoretic-degrees} uses the consistency-strength relation of definition \ref{Definition.Consistency-hierarchy}, but defines this only for sentences over the base theory, rather than theories. John Steel \cite{Steel2014:Godels-program} uses exactly the relation of definition \ref{Definition.Consistency-hierarchy}.

Stephen Simpson \cite{Simpson2009:Godel-hierarchy-and-reverse-mathematics}, however, defines the consistency-strength order as $S<T\Iff T\proves\Con(S)$. This is the  sufficiency property I mentioned earlier, which suffices for the strict order of definition \ref{Definition.Consistency-hierarchy}, but Simpson's order is not actually identical with the strict order of definition \ref{Definition.Consistency-hierarchy}---for example, Simpson's order is not dense amongst extensions of the base theory, since it can have nothing strictly between $\PA$ and $\PA+\Con(\PA)$. To my way of thinking, we should want to analyze consistency strength through a reflexive preorder relation $S\leq T$, which gives rise to the equiconsistency degrees via $S\equiv T\Iff S\leq T\leq S$, and Simpson's approach does not seem to do this. For the rest of this article, therefore, I shall proceed with the notion of consistency strength provided by definition \ref{Definition.Consistency-hierarchy}.\goodbreak

\begin{theorem}\label{Theorem.Nonlinearity}
 There are statements $\sigma$ and $\tau$ in the language of arithmetic with incomparable consistency strengths over \PA. That is, neither $\Con(\PA+\sigma)$ nor $\Con(\PA+\tau)$ provably implies the other over \PA.
\end{theorem}

\begin{proof}
Using the double version of the fixed-point lemma, we can find distinct sentences $\sigma$ and $\tau$, each asserting that for any refutation of the other sentence in the theory $\PA+\Con(\PA)$, there is a smaller refutation of itself, one with a smaller \Godel\ code. That is, $\sigma$ asserts that for any proof of $\neg\tau$ in $\PA+\Con(\PA)$, there is a smaller proof of $\neg\sigma$; and similarly vice versa with $\tau$.

Neither of these sentences, I claim, is actually refutable in $\PA+\Con(\PA)$, since if one of them were refutable, then one of them would have the smallest refutation, and this would make that sentence also provably true in \PA, which would contradict the consistency of the theory $\PA+\Con(\PA)$. So neither sentence is actually refutable and hence both are (vacuously) true.

Since $\sigma$ is not refutable, it follows that $\PA+\Con(\PA)+\sigma$ is consistent, and so it is also consistent with the assertion of its own inconsistency $\neg\Con(\PA+\Con(\PA)+\sigma)$. In any model of this combined theory, $\sigma$ is refutable in $\PA+\Con(\PA)$, but since also $\sigma$ is true there, there must not be any smaller refutation of $\tau$. Since this syntactic situation will be provable in $\PA$, it follows in light of what the sentences assert that the model thinks that $\PA$ proves that $\sigma$ is true and $\tau$ is false. So from $\Con(\PA)$ it follows both that $\Con(\PA+\sigma)$ and $\neg\Con(\PA+\tau)$ in this model.

Similarly, since $\tau$ is not refutable, we may consider the theory $\PA+\Con(\PA)+\tau$ analogously, and thereby find a model in which $\Con(\PA+\tau)$ but $\neg\Con(\PA+\sigma)$. So the two sentences have incomparable consistency strength over $\PA$, as claimed.
\end{proof}

In fact, we don't need the double fixed-point method to prove theorem \ref{Theorem.Nonlinearity}, since it is an immediate consequence of the following stronger result, which avoids the double fixed-points, while also achieving the incomparability for a sentence with its negation.

\begin{theorem}\label{Theorem.eta-neg-eta-incomparable}
 There is a statement $\eta$ in the language of arithmetic, such that $\eta$ and $\neg\eta$ have incomparable consistency strengths over \PA.
\end{theorem}

\begin{proof}
Let $\eta$ be the Rosser sentence of the theory $\PA+\Con(\PA)$, that is, the sentence asserting of itself that for any proof of $\eta$ from this theory, there is a smaller proof of $\neg\eta$, smaller in the sense of \Godel\ codes. The usual Rosser argument shows that $\eta$ is neither provable nor refutable in this theory. Namely, if $\eta$ is actually provable, then because of what $\eta$ asserts, we will have proved that $\neg\eta$ has a proof with a smaller \Godel\ code, but since the theory is consistent, none of those can be proofs of $\neg\eta$, and we can prove this; similarly, if $\eta$ were actually refutable, then we will have proved that $\eta$ is provable with a smaller code, which again would contradict the consistency of the theory.

Since $\eta$ is not provable, there is a model of $\PA+\Con(\PA)+\neg\eta$. In light of what $\eta$ asserts, this model thinks that there is a proof of $\eta$ from $\PA+\Con(\PA)$ with no smaller proof of $\neg\eta$. Since the model thinks that \PA\ can prove these concrete syntactic facts, which suffice for the failure of $\eta$, the model thinks that \PA\ proves $\neg\eta$. Since $\Con(\PA)$ holds in this model, it follows that this is a model of $\Con(\PA+\neg\eta)+\neg\Con(\PA+\eta)$.

Conversely, since $\eta$ is not refutable, it follows that $\PA+\Con(\PA)+\eta$ is consistent. By the incompleteness theorem, there is a model $M$ of this theory that also thinks this theory is inconsistent. So $M$ has a proof of $\neg\eta$ from $\PA+\Con(\PA)$, and because $\eta$ is true in the model, there must be such a proof of $\neg\eta$ with no smaller proof of $\eta$. Since \PA\ can prove this concrete syntactic fact, which furthermore suffices to validate $\eta$, it follows that the model thinks that \PA\ proves $\eta$. So from $\Con(\PA)$ in this model, we must also have $\Con(\PA+\eta)+\neg\Con(\PA+\neg\eta)$.

So we have exhibited models showing that neither the consistency statement $\Con(\PA+\eta)$ nor $\Con(\PA+\neg\eta)$ implies the other over \PA, and so these consistency strengths are incomparable.
\end{proof}

The sentences $\eta$ and $\neg\eta$ of theorem \ref{Theorem.eta-neg-eta-incomparable} exhibit what is known as the \emph{double-jump} phenomenon for consistency strength, which occurs when both a sentence and its negation jump up in consistency strength. Precisely because $\eta$ and $\neg\eta$ have incomparable consistency strengths, it follows that neither $\PA+\eta$ nor $\PA+\neg\eta$ can be equiconsistent with \PA\ alone (for then they would be comparable), and so both of them jump. Contrast this situation with the usual Rosser sentence $\rho$ defined with respect to \PA, since this sentence $\rho$ has no jumps---both $\PA+\rho$ and $\PA+\neg\rho$ are equiconsistent with \PA\ itself. Meanwhile, the \Godel\ sentence $\gamma$ has one jump, because $\PA+\gamma$ is equivalent to $\PA+\Con(\PA)$, which has strictly higher consistency strength than \PA, but $\PA+\neg\gamma$ is equivalent to $\PA+\neg\Con(\PA)$, which remains equiconsistent with \PA.

The single and double jump phenomenon is also commonly considered for the hierarchy of interpretative strength, as in \cite{Koellner2011:SEP-independence-large-cardinals}, but I'd like to mention a few differences between these hierarchies. Consistency strength is closely related to interpretative strength, to be sure, because if a theory $T$ proves $\Con(S)$, then $S$ is interpretable in $T$ with strictly lower interpretability strength simply by constructing the Henkin model; so these instances of strong increase in consistency strength are also instances of strict increase in interpretative strength. But the hierarchies are not the same. To see one difference, notice that  \cite{Koellner2011:SEP-independence-large-cardinals} points out that no $\Pi^0_1$ sentence can realize the no-jumping situation for interpretative strength, whereas we have said that the Rosser sentence $\rho$ is no-jumping in consistency strength, and this sentence has complexity $\Pi^0_1$. The difficulty for interpretation is that any model of \PA\ satisfying a $\Sigma^0_1$ statement will think that this statement must be true in all the models it interprets, since the very same existential instance is in effect inserted into the interpreted models. But this problem does not arise with consistency strength, since one can think an existential statement is consistent without yet having a specific concrete instance. With the Rosser sentence $\rho$, for example, the mere consistency of the existential case $\neg\rho$ does not cause a difficulty for the consistency of $\rho$ itself.

One can easily modify the argument of theorem \ref{Theorem.eta-neg-eta-incomparable} to look instead for proofs from $\PA+\Con(S)$, where $S$ is a consistent extension of \PA. The result is a sentence $\eta$ for which $S+\eta$ and $S+\neg\eta$ have incomparable consistency strengths. A slight generalization of this method shows in fact that the hierarchy of consistency strengths is a dense order:

\begin{theorem}
If theory $S$ has strictly weaker consistency strength than theory $T$, both extending \PA, then there is a theory $U$ of strictly intermediate consistency strength, $S<U< T$. Indeed, there are two such theories $U$ and $U'$, both strictly between $S$ and $T$, but with incomparable consistency strength to each other.
\end{theorem}

\begin{proof}
Suppose that $S$ is weaker than $T$ in consistency strength, with both theories extending \PA. It follows that $\PA+\Con(S)+\neg\Con(T)$ is consistent. Let $\delta$ be the Rosser sentence of this theory, asserting that for any proof of $\delta$ from this theory, there is a smaller proof of $\neg\delta$. The usual Rosser argument shows that $\delta$ is neither provable nor refutable in this theory.

Because $\delta$ is not provable, we have a model of $\Con(S)+\neg\Con(T)+\neg\delta$. In light of what $\neg\delta$ asserts, this model thinks there is proof of $\delta$ with no smaller proof of $\neg\delta$. Since this concrete situation can be verified by \PA, the model thinks that \PA\ and hence $S$ proves $\neg\delta$, and so this is a model of $\Con(S+\neg\delta)+\neg\Con(S+\delta)+\neg\Con(T)$.

Conversely, because $\delta$ is not refutable, the theory $\PA+\Con(S)+\neg\Con(T)+\delta$ is consistent, and so there is a model of this theory, and we may furthermore assume that the model thinks this theory is inconsistent. So this model will think that $\neg\delta$ is provable from $\PA+\Con(S)+\neg\Con(T)$, and since $\delta$ is true here, it must be that the smallest proof of $\neg\delta$ has no smaller proof of $\delta$. The model thinks that \PA\ can prove these concrete syntactic facts, which suffice to validate $\delta$, and so the model thinks that \PA\ and hence $S$ proves $\delta$. Therefore $\Con(S)$ implies $\Con(S+\delta)$ and $\neg\Con(S+\neg\delta)$ here. So this is a model of $\Con(S+\delta)+\neg\Con(S+\neg\delta)+\neg\Con(T)$.

To prove the theorem, let $U$ be the theory $T\vee(S+\delta)$, meaning the theory with all sentences of the form $\tau\vee(\sigma\wedge\delta)$, for any $\tau\in T$ and $\sigma\in S$. The models of $U$ are precisely the models of $T$ and the models of $S+\delta$, and so $\Con(U)$ is simply the disjunction $\Con(T)\vee\Con(S+\delta)$. Similarly, let $U'$ be the theory $T\vee(S+\neg\delta)$, which has $\Con(U')=\Con(T)\vee\Con(S+\neg\delta)$. It follows easily that $U$ and $U'$ are both at least weakly intermediate in consistency strength, $S\leq U,U'\leq T$.

We complete the proof by showing that $U$ and $U'$ are incomparable in consistency strength, which also implies they are both strictly intermediate between $S$ and $T$. For this, observe that the first model above had $\Con(S+\delta)$ and hence $\Con(U)$, but neither $\Con(S+\neg\delta)$ nor $\Con(T)$ and hence not $\Con(U')$. The second model, in contrast, had $\Con(S+\neg\delta)$ and hence $\Con(U')$, but neither $\Con(S+\delta)$ nor $\Con(T)$ and hence not $\Con(U)$. So neither $\Con(U)$ nor $\Con(U')$ provably implies the other, and so these theories are incomparable in consistency strength.
\end{proof}

Let me conclude this section by presenting H. Friedman's \cite{Friedman1998:Proof-theoretic-degrees} method of establishing density and incomparability in the hierarchy of consistency strength, proving it as a consequence of density in the simpler context of the derivability hierarchy, that is, in the Lindenbaum algebra, which for $\Pi^0_1$ assertions, he proves, is isomorphic to the hierarchy of consistency strength. Recall the \emph{Lindenbaum algebra} over a base theory $T$, the algebra of derivability (or direct implication), for which $\sigma\leq\tau$ when $T\proves \tau\to\sigma$. This is a Boolean algebra modulo the induced equivalence of provable equivalence. The strict order $\sigma<\tau$ holds when $\sigma\leq\tau$ but $\tau\not\leq\sigma$.

\begin{lemma}\label{Lemma.Lindenbaum-algebra-dense}
 The Lindenbaum algebra over any computably enumerable base theory $T$ extending \PA\ is dense. That is, if $\sigma<\tau$ with respect to derivability over $T$, then there is a sentence $\theta$ with $\sigma<\theta<\tau$. If $\sigma$ and $\tau$ are $\Pi^0_1$, then there is a $\Pi^0_1$ such sentence $\theta$.
\end{lemma}

\begin{proof}
Since $\sigma<\tau$, the theory $T+\sigma+\neg\tau$ is consistent. Since this provides a consistent computably axiomatizable theory of arithmetic, it is incomplete. Let $\rho$ be any statement independent of this theory, such as the Rosser sentence, which has complexity $\Pi^0_1$. Let $\theta=\sigma\wedge(\tau\vee\rho)$. So $\sigma\leq \theta\leq\tau$. But since we have models of $T+\sigma+\neg\tau+\rho$ and $T+\sigma+\neg\tau+\neg\rho$, which are therefore models of $T+\sigma+\neg\tau+\theta$ and $T+\sigma+\neg\tau+\neg\theta$, respectively, it follows that $\sigma<\theta<\tau$, as desired. Note that if $\sigma$ and $\tau$ have complexity $\Pi^0_1$, then so does $\theta$.
\end{proof}\goodbreak

Next, Friedman transfers the consistency strength hierarchy to the Lindenbaum algebra simply by observing that every sufficient $\Pi^0_1$ statement is a consistency statement.

\begin{lemma}\label{Lemma.Every-Pi01-statement-is-consistency-assertion}
 If $\theta$ is any $\Pi^0_1$ sentence with $\Con(\PA)\leq\theta$, then $\theta$ is \PA-provably equivalent to $\Con(\PA+\eta)$ for some $\Pi^0_1$ sentence $\eta$.
\end{lemma}

\begin{proof}
Using the fixed-point lemma, we can form a sentence $\eta$ that asserts of itself, ``for any counterexample to $\theta$, there is a smaller number coding a proof of $\neg\eta$.'' Assume $\PA+\theta$, and hence also $\Con(\PA)$. If $\Con(\PA+\eta)$ fails, then there is a proof of $\neg\eta$ from $\PA$. Since $\theta$ holds, this proof is coded below any counterexample to $\theta$, and so we can also prove $\eta$, contrary to $\Con(\PA)$. So $\Con(\PA+\eta)\leq\theta$. Conversely, if $\theta$ fails yet $\Con(\PA+\eta)$ holds, then there is a counterexample to $\theta$, but no smaller proof of $\neg\eta$, and from this we can prove $\neg\eta$, contrary to assumption.
\end{proof}

The argument works not just with \PA, but with any sufficient base theory. This lemma has a consequence I find remarkable, namely, that every computably enumerable theory extending the base theory is equiconsistent with an individual sentence, since $\Con(T)$ is a $\Pi^0_1$ sentence and implies $\Con(\PA)$, and so by the lemma $\Con(T)$ is equivalent to $\Con(\PA+\eta)$ for some sentence $\eta$. For example, even though the theory $\ZFC$ is not finitely axiomatizable, nevertheless there is an arithmetic sentence $\zeta$ that is equiconsistent with \ZFC\ over \PA.

Using lemma \ref{Lemma.Every-Pi01-statement-is-consistency-assertion}, Friedman deduces the following consequence, which shows how the fundamental nature of consistency strength must parallel that of direct implication.

\begin{theorem}
  The hierarchy of consistency degrees over \PA\ is isomorphic to the derivability algebra over \PA\ of $\Pi^0_1$ sentences at or above $\Con(\PA)$.
\end{theorem}

In particular, since the Lindenbaum algebra is dense by lemma \ref{Lemma.Lindenbaum-algebra-dense}, it follows that the consistency strength hierarchy also is dense; and since it is known that there are incomparable $\Pi^0_1$ statements (incomparable with respect to provable implication), it similarly follows that there are incomparable consistency strengths.

\section{Natural instances of consistency-strength nonlinearity}\label{Section.Natural-instances}

Nobody likes the examples of nonlinearity and illfoundedness provided in section \ref{Section.Formal-nonlinearity}. Those sentences are viewed as unnatural---weird self-referential logic-game trickery. Therefore, let me now embark in earnest on the hard task I have set out for myself. Everyone seems to think it is impossible, but I shall try my best.
\begin{quote}
\textbf{The hard task.} To provide natural, or at least nearly natural, instances of nonlinear incomparability and ill-foundedness in the hierarchy of consistency strength, particularly in the hierarchy of large-cardinal consistency strength.
\end{quote}

To begin, let us reflect a little on the nature of the assertions we might consider, such as the assertion that there are some finite number of inaccessible cardinals.
\begin{quote}
``There are $n$ inaccessible cardinals.''
\end{quote}
As $n$ increases, these statements increase strictly in consistency strength. But there is a subtlety here, concerning how we describe the number $n$. Must we write out $n$ as $1+1+\cdots+1$? That would be odd to insist upon if $n$ were very large.

Suppose that we had said instead that
\begin{quote}
``The number of inaccessible cardinals is at least the number of prime pairs.''
\end{quote}
This sentence surely makes a large cardinal existence assertion, but since it is an open question exactly how many prime pairs there are or whether there are infinitely many, the precise consistency strength of this assertion is not exactly clear. Even in more concrete cases such as ``there are $2^{100}$ inaccessible cardinals,'' we would ordinarily describe this number in effect by providing a method of computing it---multiply $2$ by itself $100$ times. More generally, we might want to say
\begin{quote}
``There are at least $n$ inaccessible cardinals, where $n$ is the output of this specific computational process.''
\end{quote}
Allowing such statements, however, opens the door wide to nonlinearity in the hierarchy of large cardinal consistency strength.

\begin{theorem}\label{Theorem.Incomparability-f(n)-inacc}
Amongst the large cardinal existence assertions of the form,
    \begin{quote}
      ``There are $n$ many inaccessible cardinals,''
    \end{quote}
    where $n$ is specified as the output of a specific concrete computational process, there are instances of incomparable consistency strength. Indeed, there is a computable function $f$ for which the statements
     \begin{quote}
     ``there are $f(n)$ inaccessible cardinals''
     \end{quote}
    are strongly independent with respect to consistency strength.
\end{theorem}

A list of statements is strongly independent in consistency strength when no nontrivial implication is provable between Boolean combinations of the consistency statements. This is equivalent to asserting that the consistency statements freely generate the free countably infinite Boolean algebra.

The proof of theorem \ref{Theorem.Incomparability-f(n)-inacc} will rely on an elementary variant of the universal algorithm, as follows.

\begin{theorem}\label{Theorem.Universal-computable-function}
 For any consistent theory $T$ extending \ZFC, there is a Turing machine program $e$, which we can write down, such that for any partial function $f\from\N\to\N$, there is a model of the theory, such that if we run the program $e$ on input $n$ inside the model, then for $n\in\dom(f)$ the result is $f(n)$, but if $n\notin\dom(f)$, then the program does not halt, and provably so.
\end{theorem}

\begin{proof}
The program $e$ searches for a proof from $T$ of a statement of the form ``the function computed by $e$ is not precisely the function determined entirely by these specific input/output pairs: $(k_0,n_0),\dots,(k_r,n_r)$.'' If such a proof is found, then the program proceeds to halt on exactly those inputs $k_i$ giving output $n_i$; for inputs not of the form $k_i$ on that list, the program loops endlessly. We use the Kleene recursion theorem in order to know that indeed there such a program $e$ defined by this self-referential recursion.

In the standard model of arithmetic, the program will never halt, for if it ever did halt, then it will have done so because it found such a proof that it would not halt in that way, but then proceeded to halt anyway on exactly those inputs with exactly those outputs. By inspecting the computation, we would be able to prove this is the behavior, and this would show that $T$ is inconsistent. So in the standard model, there are no such proofs to be found.

But precisely because of this, there are no such proofs to be found, and so for any particular desired finite list of input/output behavior $(k_0,n_0),\dots,(k_r,n_r)$, it must be consistent with $T$ that the program $e$ halts on these $k_i$ with output $n_i$, and diverges on all other input. Therefore, for any partial function $f\from\N\to\N$, it is finitely consistent with $T$ that the operation of $e$ is in accordance with $f$. And so the whole theory is consistent. So there is a model $M\satisfies T$ in which the function computed by $e$ is altogether in accordance with $f$ on standard input $n$.
\end{proof}

Kindly notice that if in some model of arithmetic the program $e$ should happen to halt on one input $n$ and not on another input $m$, then this will be provable by \PA\ in the model. The reason is that if the program halts on $n$, then it was because a certain proof was found involving an input/output lookup table, which did not include $m$. And so in that model \PA\ will be able to prove that the program does not halt on $m$.

Theorem \ref{Theorem.Universal-computable-function} is part of a long history, perhaps tracing back to Mostowski~\cite{Mostowski1960:A-generalization-of-the-incompleteness-theorem} and Kripke~\cite{Kripke1962:Flexible-predicates-of-formal-number-theory}. The particular proof above follows the presentation on my blog post \cite{Hamkins.blog2016:Every-function-can-be-computable}, with a key suggestion made by Vadim Kosoy in the comment section there. The theorem is mainly to be viewed, however, as a baby version of the universal algorithm theorem, due to W. Hugh Woodin \cite{Woodin2011:A-potential-subtlety-concerning-the-distinction-between-determinism-and-nondeterminism}; but see my simplified proof in \cite{Hamkins:The-modal-logic-of-arithmetic-potentialism}, and also \cite{Hamkins.blog2017:A-program-that-accepts-exactly-any-desired-set-in-the-right-universe, Hamkins.blog2017:The-universal-algorithm-a-new-simple-proof-of-Woodins-theorem}, \cite{BlanckEnayat2017:Marginalia-on-a-theorem-of-Woodin, Blanck2017:Dissertation:Contributions-to-the-metamathematics-of-arithmetic}. Shavrukov made similar arguments in 2012 private communications under the slogan ``On risks of accruing assets against increasingly better advice,'' and he pointed out connections with \cite{Berarducci1990:The-interpretability-logic-of-PA} and \cite{Japaridze1994:A-simple-proof-of-arithmetical-completeness-for-Pi11-conservativity-logic}. Albert Visser has pointed out a similar affinity with the classical proof-theoretic `exile' argument (for example, see `refugee' in~\cite{ArtemovBeklemishev2004:Provability-logic}). Further set-theoretic analogues of the universal algorithm have been explored in \cite{HamkinsWoodin:The-universal-finite-set, HamkinsWilliams2021:The-universal-finite-sequence}. The main difference between theorem \ref{Theorem.Universal-computable-function} and the full universal algorithm (as in \cite{Woodin2011:A-potential-subtlety-concerning-the-distinction-between-determinism-and-nondeterminism, Hamkins:The-modal-logic-of-arithmetic-potentialism} etc.) is the remarkable extension property of the full version. Namely, the algorithm always produces a finite sequence, but in any model of arithmetic $M$ in which the sequence produced is $s$, then for any finite extension $t$ of that sequence in $M$, there is an end-extension of $M$ to a model $N$ in which the computed sequence is $t$. Since I require only the basic property and not the extension property for the applications of this article, I have proceeded with just the baby version of theorem \ref{Theorem.Universal-computable-function}.

\begin{proof}[Proof of theorem \ref{Theorem.Incomparability-f(n)-inacc}]
Let us now use the universal computable function to establish nonlinearity in the large cardinal hierarchy. Let $f$ be the universal computable function, computed by the algorithm as described in theorem \ref{Theorem.Universal-computable-function} using any consistent theory $T$ extending \ZFC\ plus the claim that there are infinitely many inaccessible cardinals. Let $\sigma_n$ be the statement, ``there are $f(n)$ many inaccessible cardinals.'' We interpret the statement as asserting also that $f(n)$ exists. We want to show that $\Con(\ZFC+\sigma_n)$ does not \ZFC-provably imply $\Con(\ZFC+\sigma_m)$ for any $n\neq m$. For this, we need a model of $\ZFC$ satisfying the first consistency statement, but not the second. By the universality theorem \ref{Theorem.Universal-computable-function}, there is a model $M$ of $T$ in which $f(n)<f(m)$ and both are defined. Since the model $M$ has what it thinks is infinitely many inaccessible cardinals, it also has at least $f(n)+1$ many. By cutting the universe to $V_\kappa$ at the $(f(n)+1)$th inaccessible cardinal, the model thinks that the theory ``$\ZFC+f(n)<f(m)$ and there are exactly $f(n)$ many inaccessible cardinals'' is consistent. Since this theory does not prove its own consistency, there is a model $U$ inside $M$ that thinks $f^M(n)$ many inaccessible cardinals are consistent with \ZFC, but not $f^M(n)+1$ many. That model must have $\N^U$ end-extending $\N^M$, and so the operation of $e$ on $n$ and $m$ will agree with $M$, and so $f^U(n)=f^M(n)<f^U(m)=f^M(m)$. So $U$ is a model of \ZFC\ in which $\Con(\ZFC+\sigma_n)+\neg\Con(\ZFC+\sigma_m)$, as desired.
\end{proof}\goodbreak

The same argument works with Mahlo cardinals, measurable cardinals, and similarly with almost any of the usual large cardinal notions.

Perhaps one is tempted to object to theorem \ref{Theorem.Incomparability-f(n)-inacc} on the grounds that the program computing $f$ is not provably total, and this could cause ambiguity in the meaning of the statement ``there are $f(n)$ inaccessible cardinals'' if the computation of $f(n)$ does not halt. In theorem \ref{Theorem.Incomparability-f(n)-inacc}, I had interpreted the statement as false in that case. Indeed, $f(n)$ is defined only in models of $\neg\Con(T)$, which might be considered rather strange to consider if we are interested in $T$ as an aspirational large cardinal theory.

This objection is easily addressed, however, simply by using slightly different statements. Namely, let us consider the sentences asserting
\begin{quote}
``The number of measurable cardinals is at least the running time of this specific computational process.''
\end{quote}
This sentence is naturally interpreted whether or not the computational process halts; the lower bound on the number of asserted measurable cardinals will be finite if it halts and otherwise infinite. If we use the universal algorithm $e$ as described above on input $n$, then in the natural case for the die-hard large cardinal set theorist, therefore, who consider both $T$ and $\Con(T)$ to be true, the sentences will all assert the existence of infinitely many inaccessible cardinals. And yet, still they will have incomparable consistency strengths.

\begin{theorem}\label{Theorem.Incomparability-running-time}
 Amongst the large cardinal existence assertions of the form,
 \begin{quote}
   ``There are as many measurable cardinals as the running time of this specific computational process,''
 \end{quote}
 there are instances with incomparable consistency strength. Indeed, there is a computable function $f$ for which the statements about the running time of the computation of $f(n)$ are strongly independent with respect to consistency strength.
\end{theorem}

\begin{proof}
We proceed essentially the same as in theorem \ref{Theorem.Incomparability-f(n)-inacc}. Let $e$ be the program for the universal function $f$ of theorem \ref{Theorem.Universal-computable-function} defined relative to the theory $T=\ZFC+$``there are infinitely many measurable cardinals,'' and let $\tau_n$ be the statement ``there are as many measurable cardinals as the running time of $e$ on input $n$.'' If $n\neq m$, then there is a model $M\satisfies T$ in which $e$ halts on $n$, but not on $m$. If $t$ is the running time of $e$ on $n$, then we may cut off $M$ at the $(t+1)$th measurable cardinal to see that the theory ``$\ZFC+e$ halts on $n$ in $t$ steps but fails to halt on input $m$ and there are exactly $t+1$ many measurable cardinals'' is consistent. Since this theory does not prove its own consistency, there is a model $U$ inside $M$ that thinks $t$ many measurable cardinals are consistent with \ZFC, but not $t+1$ many. That model must have its natural numbers $\N^U$ end-extending $\N^M$, and so the operation of $e$ on $n$ will still  halt by time $t$ and $m$ will not. So $U$ is a model of \ZFC\ in which $\Con(\ZFC+\tau_n)+\neg\Con(\ZFC+\tau_m)$, as desired.
\end{proof}

An exactly similar argument will apply with statements of the form
 \begin{quote}
   ``There are as many supercompact cardinals as the sizes of squares in the plane that can be covered by tilings using this specific set of polygonal tiles,''
 \end{quote}
and similarly with many other such kinds of assertions, using any provably $m$-complete decision problem.

\section{Natural arithmetic instances of nonlinearity in consistency strength}\label{Section.Natural-instances-arithmetic}

A careful reader will have recognized that many of the arguments of the previous section have little actually to do with large cardinals, and the essential ideas can be used to exhibit natural instances of ill-foundedness and incomparability in the consistency strengths of simple arithmetic assertions.

\begin{theorem}\label{Theorem.Incomparability-e-halts}
 Amongst the assertions of the form,
 \begin{quote}
   ``This specific computational process halts,''
 \end{quote}
 there are instances of incomparable consistency strength. Indeed, there are instances whose consistency strengths are double-jumping. Furthermore, there is a program $e$ for which the statements ``$e$ halts on input $n$'' are strongly independent with respect to consistency strength.
\end{theorem}

\begin{proof}
Since every $\Sigma^0_1$ statement is provably equivalent to the halting of a certain computational process, the existence of double-jumping statements of this form is an immediate consequence of theorem \ref{Theorem.eta-neg-eta-incomparable}, which provided a $\Pi^0_1$ sentence $\eta$, so $\neg\eta$ is $\Sigma^0_1$.

But let me also argue directly, using the universal computable function $f$ described in theorem \ref{Theorem.Universal-computable-function}. The universal property of $f$ implies that it is consistent with any consistent strong theory $T$ that $e$ halts exactly on any desired finite set of numbers $n$ and only on those numbers. Consider any $n\neq m$. There is a model of $T$ in which $e$ halts on $n$, but not on $m$. In this case, the model thinks that it is consistent with \PA\ that $e$ halts on $n$, but inconsistent that it does so on $m$, since the reason that it halted on $n$ was because of a specific list of numbers, which did not include $m$, and in the model, \PA\ can prove that that list is what arises in the computation of $f$. So we have a model of $\PA+\Con(\PA+e\text{ halts on }n)+\neg\Con(\PA+e\text{ halts on }m)$. And similarly in the model where $e$ halts only on $m$. So these two statements have incomparable consistency strengths.

We can argue similarly that these statements are strongly independent. The consistency assertions for any nontrivial Boolean combination of the statements is determined by the halting pattern of $e$ on the finite set of numbers $m$ mentioned in the expression. And for any $n$ that is not mentioned, we can make models where the original pattern of halting is the same, except that $e$ also halts on $n$, or does not halt on $n$, respectively. So both the consistency of this statement or its negation can be added to the expression consistently, as desired.
\end{proof}

I should like specifically to call attention to the fact that the assertions used in theorem \ref{Theorem.Incomparability-e-halts} express instances of halting, rather than non-halting. It is an easy matter to show that every consistency assertion $\Con(T)$ for a computably enumerable theory $T$ is equivalent to a statement asserting that a certain program does \emph{not} halt on a certain input---we need only consider the algorithm that searches for a proof of a contradiction in $T$ and halts when such a proof is found. So the assertion that this program does not halt is provably equivalent to $\Con(T)$, and therefore has consistency strength strictly exceeding $T$.

A similar observation reveals the tight intermingling of the two meanings of `undecidable,' namely, the undecidability of a decision problem versus the undecidability of a sentence in a theory.

\begin{theorem}\label{Theorem.c.e.nondecidable-set}
Suppose that $A$ is a computably enumerable nondecidable decision problem.
 \begin{enumerate}
   \item For any consistent theory extending \PA, there are true instances of $n\notin A$ that are not provable in that theory.
   \item Furthermore, the true assertions of $n\notin A$ are not bounded in consistency strength by any consistent theory, including any consistent large cardinal hypothesis.
 \end{enumerate}
\end{theorem}

\begin{proof}
Suppose $A$ is a computably enumerable nondecidable decision problem, and consider any consistent extension $S$ of \PA. If all true instances of $n\notin A$ were provable in $S$, then we would have a decision procedure for $A$. Namely, by day we run the enumeration algorithm of $A$, learning more of the positive instances; by night, we search for proofs in $S$ that $n\notin A$, for the negative instances. Since $S$ is consistent and \PA\ proves any true instance of halting, it follows that $S$ can be trusted when it proves $n\notin A$. Since $A$ is undecidable, this process must fail as a decision procedure, and so there are $n\notin A$ for which this is not provable in the theory $S$.

Consider now any consistent computably enumerable theory $T$, such as any consistent large cardinal theory. By applying the previous observation with the theory $\PA+\Con(T)$, we see that there are instances $n\notin A$ for which this is not provable in $\PA+\Con(T)$. So there is a model of \PA\ in which $\Con(T)+(n\in A)$. Since $n\in A$ implies that \PA\ proves it, the model thinks that $\PA+(n\notin A)$ is inconsistent. So this is a model of \PA\ in which $T$ is consistent, but $\PA+(n\notin A)$ is not, showing that the consistency strength of the assertion $n\notin A$ is not bounded by $T$, as claimed.
\end{proof}

Thus, every computably enumerable computably undecidable set is saturated with logical undecidability and nontrivial consistency strength. By (2), the true assertions of $n\notin A$ can have no largest instance of consistency strength. From this, it follows that either there is incomparability in consistent strength amongst these statements, or else the statements form a linear hierarchy of consistency strength, not bounded above by any given consistency strength.

Recall that a computably enumerable set $A\of\N$ is said to be \emph{$m$-complete}, if for every computably enumerable set $B$ there is a computable total function $f:\N\to\N$ such that $b\in B\iff f(b)\in A$. For example, the halting problem, the word problem for groups, the nontiling problem, and many other problems are well known to be $m$-complete. Meanwhile, observation \ref{Observation.Provably-m-complete} below will show that it is a strictly stronger hypothesis about a program that it enumerates a \emph{\PA-provably} $m$-complete set, and strictly stronger still to assume that the enumerated set admits specific provable reductions $f$ from the familiar $m$-complete sets.

\begin{theorem}\label{Theorem.Incomparable-m-complete}
 Suppose that $A$ is any computably enumerable $m$-complete decision problem.
 \begin{enumerate}
   \item If $A$ is \PA-provably $m$-complete, with a provable reduction of the halting problem, then amongst the true statements $n\notin A$, there are instances strictly exceeding any given consistency strength.
   \item But even without that extra provability assumption, amongst the true assertions $n\notin A$ there are instances with incomparable consistency strength, and instances of double-jumping consistency strength.
   \item Amongst the consistent statements $n\in A$, there are instances of incomparable consistency strength and of double-jumping consistency strength.
   \item Within both kinds of statements $n\notin A$, $n\in A$, there are effective enumerations that are strongly independent in consistency strength.
 \end{enumerate}
\end{theorem}

\begin{proof}
For statement (1) only, we assume that $A$ is provably $m$-complete, with a provable reduction from the halting problem. For any computably enumerable theory $T$, let $e_T$ be the program that searches for a proof of a contradiction in $T$, halting only when found. By the reduction of the halting problem to $A$, we get a number $n$ such that $e_T$ halts if and only if $n\in A$, provably in \PA. So $\Con(T)$ is provably equivalent to the assertion $n\notin A$, which therefore has consistency strength strictly exceeding $T$.

For the rest of the statements, let us now drop the extra assumption about provable $m$-completeness. We now assume only that $A$ is actually computably enumerable and $m$-complete. Let $\pi$ be a specific computable function that is a reduction to $A$ of the halting problem. Consider a version of the universal computable function $f$, as in theorem \ref{Theorem.Universal-computable-function}, but defined relative to the theory $\PA+\Con(\ZFC)+$``$\pi$ is a reduction of the halting problem to $A$.'' The proof of theorem \ref{Theorem.Universal-computable-function} shows that there is a computable procedure $e$ such that for any $n\neq m$, there is a model of this theory in which $e$ halts on $n$ but not $m$, and another in which $e$ halts on $m$ and not $n$. In the first case, we'll have $\pi(e,n)\in A$, $\pi(e,m)\notin A$ and in the second case vice versa. Because of the way $e$ is defined, if it halts on one number but not another, then in such a model it will be thought inconsistent that it could halt on the other number. Thus, the statements of the form $\pi(e,n)\notin A$ are incomparable in consistency strength, and indeed  these statements form a strongly independent family with respect to consistency strength, just as in theorem \ref{Theorem.Incomparability-e-halts}. Furthermore, any one of these statements is double-jumping, because, once $\pi(e,n)$ is in $A$, then it is inconsistent for it to be out of $A$, and if it is out of $A$, but another number is in $A$, then it is inconsistent for it to be in $A$. This establishes statement (2), and the part of $4$ for assertions of the form $n\notin A$.

The rest of the claims, statement (3) and the part of (4) for $n\in A$, are proved by essentially the same argument, which exhibits a certain symmetry as to whether we were asserting membership or nonmembership of $\pi(e,n)$ in $A$.
\end{proof}

In light of the fact that almost all of the undecidable decision problems commonly identified as ``natural'' are also provably m-complete, theorem \ref{Theorem.Incomparable-m-complete} shows that these decision problems (and their complements) are saturated with instances of incomparable consistency strengths.

\begin{corollary}\ \label{Corollary.Tiling-problems}
 \begin{enumerate}
   \item Amongst assertions of the form
    \begin{quote}
      ``This specific set of polygonal tiles admits a tiling of the plane''
    \end{quote}
    there are consistent instances strictly exceeding any given consistency strength.
   \item There are such assertions with incomparable consistency strength.
   \item There are such assertions having incomparable consistency strength with their own negations---both the statement and its negation jump in consistency strength.
   \item There is an effective enumeration of finite tile sets $t_0, t_1, \dots$, such that the assertions ``tile set $t_n$ admits a tiling of the plane'' have strongly independent consistency strengths.
 \end{enumerate}
\end{corollary}

\begin{proof}
The tiling problem is provably $m$-complete, with a provable reduction of the halting problem.
\end{proof}\goodbreak

We can just as easily provide corresponding results for statements of the form
\begin{quote}
``this specific diophantine equation $p(\vec x)=0$ has no solution in the integers,''
\end{quote}
where $p$ is explicitly provided as a polynomial over the integers. Such statement are not bounded in consistency strength by any consistent theory, and there are specific instances with incomparable consistency strength and double-jumping consistency strength. There is a particular integer polynomial $p(\vec x,y)$ for which the statements ``$p(\vec x,n)=0$ has a solution in the integers,'' as $n$ varies, have strongly independent consistency strengths.

We can provide similar results for assertions such as
\begin{quote}
``this specific finite group presentation is the trivial group''
\end{quote}
or
\begin{quote}
``this cell in this specific Game of Life position will eventually become alive.''
\end{quote}
And so on. In each case, there are specific such statements with incomparable consistency strength, with double-jumping consistency strength, and schemes of such statements with strongly independent consistency strength.

What I take these arguments to show is that there is pervasive nonlinearity at every level of the consistency-strength hierarchy, and this is a direct consequence of the difficulty of interpreting even the names of natural numbers. We can describe numbers, even specifying them by concrete computable procedures, but our base theory may just not settle the question of which number is larger, and this leads directly and almost immediately to incomparable consistency strengths.

Meanwhile, regarding the difference between being $m$-complete and being provably $m$-complete, or having provable reductions of specific computably enumerable sets, let me prove that indeed these are not the same.

\begin{observation}\label{Observation.Provably-m-complete}\
 \begin{enumerate}
   \item There are $m$-complete computably enumerable sets that are not \PA-provably $m$-complete.
   \item There are \PA-provably $m$-complete sets $A$, for which none of the usual $m$-complete sets $B$ have a specific reduction function that is \PA-provably a reduction of $B$ to $A$.
 \end{enumerate}
\end{observation}

\begin{proof}
The halting problem is, of course, a \PA-provably $m$-complete computably enumerable set. Consider the computably enumerable set $A$, which enumerates the halting problem, unless it happens to find a proof of a contradiction in some fixed strong theory $T$, in which case it enumerates every number into $A$, spoiling the set. If $T$ is actually consistent, then this algorithm will enumerate the halting problem, which is $m$-complete. But this set is not \PA-provably complete, since it is consistent with \PA\ that $\neg\Con(T)$, in which case the set will be all of $\N$ and therefore not $m$-complete. This proves statement (1).

For statement (2), let us simply modify the previous algorithm, so that when the proof of a contradiction in $T$ is found, the program enumerates all numbers up to the code of that proof into the set, but then starts over above this point by enumerating a shifted copy of the halting problem. In any model of \PA, this set will either be the halting problem or a shifted copy of the halting problem (with a filled-in block below), and hence $m$-complete. So we can prove in \PA\ that this algorithm enumerates an $m$-complete computably enumerable set. But there will be no particular reduction function from the usual halting problem (or any of the other commonly considered $m$-complete sets) that we can prove in \PA\ is a reduction to $A$, since it is consistent with \PA\ that every standard number is in $A$, and so we will get wrong answers for the instances of provable nonhalting programs in such a model.
\end{proof}

It it very natural to inquire whether one can prove theorem \ref{Theorem.Incomparable-m-complete} under the weaker assumption only that $A$ is computably enumerable and undecidable, rather than $m$-complete. Theorem \ref{Theorem.c.e.nondecidable-set} gets part of the result, but does not achieve incomparability in the consistency assertions.

\begin{question}
 Does every computably enumerable noncomputable set $A$ admit statements $n\notin A$ of incomparable consistency strength?
\end{question}

The answer is no, proved by Uri Andrews after a talk I gave on the topics of this article for the Madison Logic Seminar.

\begin{theorem}[Uri Andrews]\label{Theorem.c.e.-undecidable-but-linear}
 There is a c.e.~undecidable set $A$, which is Turing equivalent to the halting problem, but for which the assertions $n\notin A$ are linearly ordered by consistency strength. Also the assertions $n\in A$ are linearly ordered by consistency strength.
\end{theorem}

\begin{proof}
For any c.e.~undecidable set $I\of\N$, let $r=\sum_{n\in I}\frac{1}{2^n}$ be the real number having binary bit $1$ in the places of $I$. Let $A$ be the rational cut determined by $r$, the set $\set{q\in\Q\mid q<r}$. Let us assume we have naturally encoded rational numbers with natural numbers. The set $A$ is c.e., because as natural numbers are enumerated into $I$, we learn improved lower bounds to $r$ and can enumerate the corresponding rational numbers into $A$. The set $A$ is Turing equivalent to $I$, since with $I$ as an oracle we can compute $r$ well enough so as to answer definitively about any given rational number; and conversely with $A$ as an oracle we can compute $I$. So by taking $I$ to be Turing complete we shall have $A$ Turing equivalent to the halting problem.

The key thing to notice about $A$, however, is that for rational numbers, if $p<q$ and $q\in A$, then $p\in A$. Equivalently, $p<q$ and $p\notin A$ implies $q\notin A$, and this will be provable in \PA. So there is a provable linear relation about any two nonmembers of $A$. The consequence of this is that the assertions $p\notin A$ and $q\notin A$ cannot be incomparable in consistency strength, since any model of the former is necessarily a model of the latter, and so $\Con(\PA+p\notin A)$ implies $\Con(\PA+q\notin A)$. So the nonmembership assertions $p\notin A$ are linearly ordered by consistency strength. An essentially similar argument works with the positive assertions $p\in A$.
\end{proof}

In light of theorem \ref{Theorem.c.e.nondecidable-set}, it follows that the consistency strengths of the assertions $p\notin A$ for the set $A$ mentioned in the proof of theorem \ref{Theorem.c.e.-undecidable-but-linear} form a linear order of height $\omega$. And the consistency strengths of the statements $p\in A$ form a linear hierarchy of order type $\omega^*$.

\section{Natural instances of ill-foundedness and incomparability via cautious theory enumerations}\label{Section.Cautious-theories}

Let me now present another large class of examples. Imagine that we believed in a certain computably enumerable theory, such as \PA, \ZFC, or \ZFC\ plus large cardinals, whose axioms we intended to enumerate, but \emph{cautiously}. By doing so, we can often enumerate the theory in a sensible, natural manner, but with a strictly lower consistency strength or with incomparable consistency strengths. These cautious enumeration theories will thereby provide natural instances both of ill-foundedness and nonlinearity in the hierarchy of consistency strength.

As a general umbrella term, I shall say that one has a \emph{cautious enumeration} of a theory, if one is enumerating the axioms of the theory according to a procedure that will continue as long as one hasn't encountered a certain kind of contrary indicator, a reason to doubt the truth of the theory. Different cautious enumerations of a theory will arise depending on the specific contrary indicator that is required to halt the enumeration.

I shall also use the term to refer to specific natural theories. For example, let the \emph{cautious enumeration of \ZFC} be the enumeration of \ZFC\ that continues as long as we have not yet found a proof in what we have enumerated so far that \ZFC\ is inconsistent. I denote the resulting theory by $\ZFCcautious$. 
In order to halt the enumeration, notice that we do not require an explicit contradiction in \ZFC, but rather only a proof that there is such a contradiction. If $\Con(\ZFC)$ is actually true, as we have assumed, and furthermore, if $\Con(\ZFC)$ is consistent with \ZFC, then no confounding proof of inconsistency will ever be found. As a theory, therefore, the cautious enumeration $\ZFCcautious$ will actually have all the same axioms as \ZFC; they are the same theory, but with a different manner of enumeration. This situation will reveal the intentional aspect of the consistency strength hierarchy that I discussed earlier, for the consistency strength of the cautious theory $\ZFCcautious$ will be strictly less than $\ZFC$, even when these theories have the same sentences.

I find the cautious enumeration to be both sensible and realistic---in this sense it is a natural theory---for if we were actually enumerating \ZFC\ and a proof was pointed out to us along the way that the theory we have already committed to proves the full \ZFC\ theory inconsistent, then we would have ample reason to pause and reflect on whether we should continue with the enumeration. I think we would pause the enumeration right then and reconsider. Thus, the cautious enumeration is what we would actually do---and for this reason, I argue, it is a natural theory.

Philosophers of logic have written at length on the principle by which whenever we are inclined to accept a theory $T$, then we should also be inclined to accept the assertion $\Con(T)$ expressing that this theory is consistent. Surely there would be something odd or unnatural about accepting a theory $T$ but refusing to accept $\Con(T)$. In short, if we find it natural to accept theory $T$, then we should also find it natural to accept $\Con(T)$.

The cautious enumerations amount at bottom to the contrapositive of this. Namely, if in the process of enumerating a theory $T$ one observes that it would be provably inconsistent, that is, if one finds a proof of $\neg\Con(T)$ in $T$, then according to the principle we should refrain from finding $T$ acceptable. And this is precisely what the cautious enumeration does. Thus, for this second reason, the cautious enumeration is a natural theory.

\begin{theorem}
 The cautious enumeration $\ZFCcautious$ is an alternative computable enumeration of \ZFC, with a strictly lower consistency strength than \ZFC.
\end{theorem}

\begin{proof}
We have already explained that if $\Con(\ZFC)$ is consistent with \ZFC, then the cautious enumeration $\ZFCcautious$ will never encounter the confounding proof of inconsistency, and so it will fully enumerate all the \ZFC\ axioms. Meanwhile, I claim that the cautious theory $\ZFCcautious$ has a strictly weaker consistency strength than \ZFC. We can easily prove, of course, that $\ZFCcautious\of\ZFC$ and consequently that $\Con(\ZFC)\to\Con(\ZFCcautious)$. But if $\ZFC+\Con(\ZFC)$ is consistent, then there is a model of the theory $\ZFC+\Con(\ZFC)+\neg\Con(\ZFC+\Con(\ZFC))$. In such a model $M$, while $\ZFC$ is consistent, nevertheless there will be a proof from \ZFC\ that it is not, and so $M$ will think that $\ZFCcautious$ consists of finitely many axioms of $\ZFC$. Since we can prove even in \PA\ that \ZFC\ proves each of its finite subsets is consistent, it follows that $M$ thinks that $\ZFC$ proves $\Con(\ZFCcautious)$. By the incompleteness theorem, it thinks $\ZFC+\neg\Con(\ZFC)$ is consistent, and so it can build a model of \ZFC\ in which $\Con(\ZFCcautious)$ holds, but not $\Con(\ZFC)$. So the cautious theory $\ZFCcautious$ is strictly weaker than \ZFC\ in consistency strength, even though they enumerate the same theory.
\end{proof}

It might be a little surprising that a theory can have a strictly weaker consistency strength over a base theory than the base theory itself. But it shouldn't be too surprising, since there is room between $\ZFC$ and $\ZFC+\Con(\ZFC)$ for other consistency statements, and that is what is going on here. The theorem amounts to the claim that $\ZFC+\Con(\ZFCcautious)$ is a strictly weaker theory than $\ZFC+\Con(\ZFC)$.

Consider next the \emph{doubly cautious} enumeration $\ZFC^{\circ\circ}$, where we enumerate the \ZFC\ axioms as usual, but continue only as long as we have not yet found a proof in \ZFC\ that \ZFC\ is inconsistent or even a proof in \ZFC\ that there is such a proof of inconsistency. In other words, we stop the enumeration when we find a proof from \ZFC\ either of $\neg\Con(\ZFC)$ or of $\neg\Con(\ZFC+\Con(\ZFC))$.

\begin{theorem}
 The doubly cautious enumeration of \ZFC\ is an alternative computable enumeration of \ZFC, with strictly weaker consistency strength than even the cautious enumeration.
\end{theorem}

\begin{proof}
Since $\ZFC^{\circ\circ}$ is a subtheory of $\ZFCcautious$, we easily prove in \ZFC\ that $\Con(\ZFCcautious)$ implies $\Con(\ZFC^{\circ\circ})$. To show the implication is not reversible, consider a model of $\ZFC+\Con(\ZFC)+\Con(\ZFC+\Con(\ZFC))+\neg\Con(\ZFC+\Con(\ZFC+\Con(\ZFC)))$. In this model, because $\Con(\ZFC+\Con(\ZFC))$ holds, the cautious enumeration will never find a proof of $\neg\Con(\ZFC)$ from \ZFC, and so it will enumerate the full theory, $\ZFCcautious=\ZFC$. But because the models believes $\neg\Con(\ZFC+\Con(\ZFC+\Con(\ZFC)))$, it will find a proof from \ZFC\ of $\neg\Con(\ZFC+\Con(\ZFC))$. In other words, it will find a proof from \ZFC\ that there is a proof from \ZFC\ that \ZFC\ is inconsistent. Thus, the doubly cautious enumeration will find its stopping point, and so $M$ thinks $\ZFC^{\circ\circ}$ is a finite fragment of $\ZFC$. So it thinks that $\ZFCcautious$ proves $\Con(\ZFC^{\circ\circ})$. Since by $\Con(\ZFC)$ it thinks $\ZFCcautious+\neg\Con(\ZFCcautious)$ is consistent, it can therefore build a model of \ZFC\ in which $\Con(\ZFC^{\circ\circ})$ holds, but not $\Con(\ZFCcautious)$. So the doubly cautious theory $\ZFC^{\circ\circ}$ is strictly weaker in consistency strength than the cautious theory $\ZFCcautious$.
\end{proof}

We can obviously iterate this to the triply cautious enumeration, and so on, resulting in an effective list of alternative enumerations of \ZFC, all enumerating exactly the same full \ZFC\ theory, but doing so with progressively weaker consistency strengths.
 $$\cdots<\ZFC^{\circ\circ\circ}<\ZFC^{\circ\circ}<\ZFCcautious<\ZFC$$
Thus, we have ill-foundedness in consistency strength even amongst the enumerations of the same \ZFC\ theory. And all these enumerations are quite sensible, since if we believe in \ZFC, then the discovery that it is provably inconsistent or that it is provably provably inconsistent would surely make us pause and reflect. So I find this to be a natural instance of ill-foundedness in the hierarchy of consistency strength.

One might also be curious about another slightly less cautious enumeration of \ZFC, what I call the \emph{stop-when-hopeless} enumeration $\ZFC^\otimes$, where we enumerate \ZFC\ as usual, but continue adding axioms only so long as we have not found any explicit contradiction yet in what has been enumerated. If \ZFC\ is consistent, then no such contradiction will ever be found, and so the resulting theory $\ZFC^\otimes$ has all the same axioms as \ZFC. Is this theory also strictly weaker in consistency strength than \ZFC? The answer is no, I claim, the stop-when-hopeless enumeration $\ZFC^\otimes$ is equiconsistent with \ZFC. Since $\ZFC^\otimes$ is provably a subtheory of $\ZFC$, we can easily prove that $\Con(\ZFC)\to\Con(\ZFC^\otimes)$. Conversely, if we are in a model of \ZFC\ in which $\ZFC^\otimes$ is consistent, then the enumeration process must not ever have stopped, and so it is the same theory as \ZFC. So in \ZFC, we can prove $\Con(\ZFC)\iff\Con(\ZFC^\otimes)$. So the stop-when-hopeless enumeration offers no advantage in consistency strength.

In general, for any theories $S$ and $T$, we may consider the \emph{$T$-cautious enumeration of $S$}, where we enumerate the axioms of $S$ as long as we have not yet seen an explicit contradiction in $T$. In this terminology, $\ZFCcautious$ is the $\ZFC+\Con(\ZFC)$-cautious enumeration of \ZFC, and $\ZFC^{\circ\circ}$ is the $\ZFC+\Con(\ZFC)+\Con(\ZFC+\Con(\ZFC))$-cautious enumeration of \ZFC. The stop-when-hopeless enumeration $\ZFC^\otimes$ is simply the \ZFC-cautious enumeration of \ZFC.

Meanwhile, let me now offer several additional alternative cautious enumerations of \ZFC, each looking for a slightly different particular reason to halt the enumeration, each of them sensible and compelling. Yet the resulting theories have strictly incomparable consistency strengths, even though they all actually enumerate the same \ZFC\ theory.

\begin{theorem}
 There are diverse alternative cautious enumerations of the \ZFC\ theory, of incomparable consistency strengths to one another and hence all strictly weaker than \ZFC. Indeed, there is an effective enumeration of theories $\ZFC^{(n)}$, each actually having exactly the same axioms as \ZFC, but whose consistency assertions are strongly independent, generating the countable free Boolean algebra.
\end{theorem}

\begin{proof}
Let $f$ be the universal computable function defined relative to the theory $\ZFC+\Con(\ZFC)$. This algorithm searches for a proof in that theory that $f$ does not agree exactly with a certain finite list of input/output pairs, and when found, it computes exactly in accordance with that list, diverging on all other input. Let $\ZFC^{(n)}$ be the theory that continues enumerating the \ZFC\ axioms as long as $f(n)$ has not yet halted.

Each $\ZFC^{(n)}$ is sensible, each cautious in a slightly different and independent way, for each of them enumerates the \ZFC\ axioms provided a certain kind of contrary information is not encountered. Namely, the $n$th theory $\ZFC^{(n)}$ is watching to see whether $f(n)$ halts, and if it does, then this is a perfectly sound reason to doubt the veracity of \ZFC, because the halting of $f(n)$ occurs only when there is a proof that $f$ doesn't behave in a manner that it has been observed to behave, showing that $\ZFC+\Con(\ZFC)$ has false consequences; on these grounds, it is natural for $\ZFC^{(n)}$ to stop its enumeration. Meanwhile, since the program $f$ does not actually ever halt (in the standard model $\N$), each theory $\ZFC^{(n)}$ actually has all the same axioms as \ZFC. The disagreements between these theories are merely theoretical possibilities, realized only possibly in a nonstandard model and even then only on nonstandard-length axioms.

The theories are bounded above in consistency strength by \ZFC, simply because each $\ZFC^{(n)}$ is provably a subtheory of \ZFC. But indeed, I claim that they are each strictly strictly weaker than \ZFC\ in consistency strength, for if $n\neq m$, then I shall now prove that $\ZFC^{(n)}$ and $\ZFC^{(m)}$ have incomparable consistency strength over \ZFC. Assume $n\neq m$. By the properties of the universal computable function, there is a model $M$ of $\ZFC+\Con(\ZFC)$ in which $f(n)$ halts but $f(m)$ does not. In this model, $\ZFC^{(n)}$ is a finite fragment of \ZFC, but $\ZFC^{(m)}$ is fully the same as $\ZFC$. Since $M$ thinks $\Con(\ZFC)$, it can build a model $N$ that it thinks satisfies $\ZFC+\neg\Con(\ZFC)$. Because the function $f$ halted at a finite stage of $M$, the model $N$ will agree with this, since $\N^M$ must be an initial segment of $\N^N$. So the computation of $f$ on inputs $n$ and $m$ is the same in $N$ as it is in $M$. Since we can prove even in \PA\ that the theory \ZFC\ proves the consistency of any of its finite fragments, it follows that $M$ thinks $\Con(\ZFC^{(n)})$ holds in $N$. But since $\neg\Con(\ZFC)$ also holds and $f(m)$ does not halt, it follows that $N$ satisfies $\neg\Con(\ZFC^{(m)})$. So we have a model of \ZFC\ which thinks $\ZFC^{(n)}$ is consistent, but not $\ZFC^{(m)}$. This shows that the consistency of $\ZFC^{(n)}$ does not provably imply the consistency of $\ZFC^{(m)}$ over \ZFC, and so all the theories have incomparable consistency strength.

It is not much more difficult to show that these consistency statements are strongly independent and freely generate the free Boolean algebra. The point is that for any conjunction of the statements $\Con(\ZFC^{(n)})$ or their negations, then we can find a model where $f$ halts on the $n$s for which $\Con(\ZFC^{(n)})$ occurs positively in the conjunction and not on the $n$s that occur negatively, and this will make the conjunction altogether true; or conversely we can find a model making it false. By controlling which $n$ the universal algorithm halts on, we can control exactly which consistency statements are strong and which are weak, showing that every possible combination of those consistency statements is nontrivial. So these statements are strongly independent and therefore freely generate the free Boolean algebra.
\end{proof}

Let $I$ be the assertion that there is an inaccessible cardinal.

\begin{theorem}
 The cautious enumeration of $\ZFC+I$ has consistency strength strictly between \ZFC\ and $\ZFC+I$. Indeed, there is an effective enumeration of infinitely many cautious enumerations of this theory, with strongly independent incomparable consistency strengths, whose consistency statements freely generate the countable free Boolean algebra.
\end{theorem}

\begin{proof}
Consider first the cautious enumeration $(\ZFC+I)^\circ$ of the theory $\ZFC+I$, where we enumerate the axioms into this theory, including $I$ itself, up until a stage at which we see a proof in $\ZFC+I$ that $\ZFC+I$ is inconsistent---let us suppose that we include Zermelo set theory plus $I$ and $\Sigma_n$ replacement for all $n$ up to the size of a proof of $\neg\Con(\ZFC+I)$. I should like to emphasize again that we don't insist on seeing an actual proof of a contradiction from this theory, but rather merely a proof that there is one. Since we assume $\ZFC+I+\Con(\ZFC+I)$ is consistent, this cautious enumeration will provide all of the axioms of the full theory $\ZFC+I$. The difficulty is that we are not able to prove this in \ZFC---we had made an extra consistency assumption to deduce it.

The theory $\ZFC+I$ proves $\Con(\ZFC)$, since we can prove that if $\kappa$ is inaccessible, then $V_\kappa$ is a transitive model of \ZFC. Since this proof can be made explicit, the proof can be made using axioms that appear before any supposed (nonstandard) proof that there is a contradiction. Therefore, $\Con(\ZFC)$ will be provable in $(\ZFC+I)^\circ$, provably so, and so this cautious theory is strictly stronger than $\ZFC$ in consistency strength. This argument extends to the iterated consistency statements $\Con(\ZFC+\Con(\ZFC))$ and so on.

Meanwhile, since we have assumed that $\ZFC+I+\Con(\ZFC+I)$ is consistent, by the incompleteness theorem it is consistent with $\neg\Con(\ZFC+I+\Con(\ZFC+I))$. In a model $M$ of this theory, there will come a stage where $\ZFC+I$ proves $\neg\Con(\ZFC+I)$, which will stop the cautious enumeration. So in $M$, the theory $(\ZFC+I)^\circ$ is a finite fragment of $\ZFC+I$. But we can prove in a weak meta theory that $\ZFC+I$ proves the consistency of each of its finite fragments, and so this model thinks that $\ZFC+I$ proves $\Con((\ZFC+I)^\circ)$. Since $M$ has a model of $\ZFC+I+\neg\Con(\ZFC+I)$, it will therefore have a model of $\ZFC+I+\Con((\ZFC+I)^\circ)+\neg\Con(\ZFC+I)$. Therefore, we cannot prove in \ZFC, nor even in $\ZFC+I$, that the consistency of the cautious theory $(\ZFC+I)^\circ$ implies the consistency of $\ZFC+I$. So it has strictly intermediate consistency strength.

For the incomparability result, let $f$ be the universal computable function, defined relative to the theory $\ZFC+I+\Con(\ZFC+I)$. That is, $f$ looks for a proof in this theory that the input/output pattern of $f$ does not agree with some explicit finite list, and when found, the function $f$ halts or diverges on its input exactly in accordance with that list. Let $(\ZFC+I)^{(n)}$ be the cautious enumeration in which one enumerates the axioms of $\ZFC+I$ until a stage at which $f(n)$ halts, at which point the enumeration also halts. This is sensible, since the halting of $f(n)$ is a sensible reason not to trust $\ZFC+I$, since the assertion that it was consistent led to false statements.

Meanwhile, I claim these theories are incomparable in consistency strength, and strongly independent. For any pair $n\neq m$, there is a model $M\satisfies\ZFC+I+\Con(\ZFC+I)$ in which $f(n)$ halts but $f(m)$ does not, and provably so. In such a model, the theory $(\ZFC+I)^{(n)}$ is a finite fragment of $\ZFC+I$, whereas $(\ZFC+I)^{(m)}$ is provably the full theory $\ZFC+I$. By the same method as earlier, therefore, we can find a model of $\ZFC+I$ in which $\Con((\ZFC+I)^{(n)})$ holds, but $\Con((\ZFC+I)^{(m)})$ does not. So the two theories have incomparable consistency strength, even over $\ZFC+I$. Similar reasoning shows that these theories are strongly independent.
\end{proof}

A similar analysis applies to the axioms asserting more than one inaccessible cardinal. Let $I_n$ be the assertion that there are $n$ inaccessible cardinals, where $n$ is any natural number, and more generally, let $I_\alpha$ be the assertion that there are at least $\alpha$ many inaccessible cardinals, in their natural order, for any ordinal $\alpha$. Consider the cautious assertion of infinitely many inaccessible cardinals, namely the axiom $I_{\omega^\circ}$ asserting that there are $n$ inaccessible cardinals for every natural number $n$ up to the size of the smallest proof of a contradiction, if any, in the theory $\ZFC+I_\omega$, which asserts that there are infinitely many inaccessible cardinals. Note that the assertion $I_{\omega^\circ}$ is a single axiom, in contrast with the cautious enumerations of theories we considered earlier.

\begin{theorem}
 The cautious assertion of infinitely many inaccessible cardinals is strictly stronger than any particular finite number of inaccessible cardinals and strictly weaker than infinitely many.
 $$I<I_2<I_3<\quad\cdots\quad<I_{\omega^\circ}<I_\omega$$
\end{theorem}

\begin{proof}
We have assumed that it is consistent with \ZFC\ for there to be infinitely many inaccessible cardinals, and so the theory $\ZFC+I_{\omega^\circ}$ will prove every particular $I_n$ as a theorem. So the consistency strength of $I_{\omega^\circ}$ will strictly exceed that of any particular $I_n$. On the other hand, there is a model $M$ satisfying $\ZFC+I_\omega+\neg\Con(\ZFC+I_\omega)$. This model will think that $I_{\omega^\circ}$ stops the assertions at some (nonstandard) $n$, the stage at which $\ZFC$ proves a contradiction from the theory $\ZFC+I_\omega$. The model thinks $I_{n+1}$ is true, and hence that $\Con(\ZFC+I_n)$ and consequently that $\Con(\ZFC+I_{\omega^\circ})$ is true. But since it thinks $\neg\Con(\ZFC+I_\omega)$, this shows the cautious axiom is strictly weaker in consistency strength.
\end{proof}

A similar idea will now enable us to exhibit a similar natural family of strongly independent axioms, incomparable in consistency strength. Let $f$ be the universal computable function, defined relative to the theory $\ZFC+I_\omega$. Let $I_{\omega^\circ}^{(k)}$ be the axiom asserting that the number of inaccessible cardinals is at least the running time of $f(k)$. If $f(k)$ halts, after all, this is a reason to distrust the theory $\ZFC+I_\omega$, which asserts infinitely many inaccessible cardinals, since it is an instance where this theory proved something observably false, and so in this case $I_{\omega^\circ}^{(k)}$ retreats to assert only some finitely many inaccessible cardinals. But just as proved in theorem \ref{Theorem.Incomparability-running-time}, it will turn out that these particular reasons for limiting the full axiom $I_\omega$ are strongly independent over $\ZFC+I_\omega$. Each axiom $I_{\omega^\circ}^{(k)}$ is being cautious about $I_\omega$ in a different, independent manner; but every one of them is quite reasonable.

\begin{theorem}
The cautious inaccessibility assertions $I_{\omega^\circ}^{(k)}$ are pairwise incomparable in consistency strength over \ZFC\ and indeed strongly independent---their consistency statements freely generate the countable free Boolean algebra. Each of these axioms is strictly stronger than every $I_n$, but strictly weaker than $I_\omega$.
\end{theorem}

\begin{proof}
If $k\neq r$, then there is a model of $\ZFC+I_\omega$ in which $f(k)$ is defined and $f(r)$ is not, and the model thinks that \ZFC\ can prove this. The model must also think $\neg\Con(\ZFC+I_\omega)$, and so this is a model in which $\ZFC+I_{\omega^\circ}^{(k)}$ is thought to be consistent and $\ZFC+I_{\omega^\circ}^{(r)}$ is not. A similar argument works with any finite pattern for the consistency of the theories $I_{\omega^\circ}^{(k)}$, and so these axioms are strongly independent in consistency strength.
\end{proof}

Another way to describe the cautious theories $I_{\omega^\circ}^{(k)}$ is that they each assert that there are at least $n$ inaccessible cardinals for every $n$ up to a certain kind of observed violation of $\Pi_1$ soundness for the theory $\ZFC+I_\omega$, if any, a case where this theory proves a certain $\Pi^0_1$ assertion about $k$ for which a counterexample also is observed. The point is that the halting of $f(k)$ is just like this: a proof is found that $f$ doesn't have a certain behavior that it is also observed immediately to exhibit. Would you trust a theory that proved a $\Pi^0_1$ statement for which you are observing a counterexample? I think you would find it natural not to trust such a theory, and in this sense, the cautious axioms $I_{\omega^\circ}^{(k)}$ are natural.

One might have been tempted to try to formulate a cautious version of $I_\omega$ by asserting that there are as many inaccessible cardinals as are consistent, up to $\omega$. That is, the assertion that for every $n$, if $\ZFC+I_n$ is consistent, then $I_n$. This statement, however has no consistency strength at all---it is equiconsistent with \ZFC---because it is equiconsistent with \ZFC\ that \ZFC\ itself is inconsistent, which trivializes the stated axiom. The cautious formulation $I_{\omega^\circ}$, in contrast, strictly exceeds every particular $I_n$.

Generalizing the previous temptation, let us define the \emph{provisional} assertion of any sentence $A$ to be the assertion $\Con(\ZFC+A)\to A$. This sentence asserts that $A$ is true, provided that it is consistent. But again, it has no consistency strength at all over \ZFC, because it is relatively consistent with \ZFC\ that $\neg\Con(\ZFC)$, which trivializes the provisional assertion of $A$.

We also have the relatively cautious theories, as before, such as the theory asserting that there are as many inaccessible cardinals as the size of the smallest proof that there is no measurable cardinal, or infinitely many if there is no such proof. That is, we assert more and more inaccessible cardinals, up to the size of the smallest proof that measurable cardinals are inconsistent. The key observation is that as we strengthen measurable to strong and supercompact, and so on, then the consistency strength of the relatively cautious theories descend, providing natural instances of illfoundedness in the hierarchy.

One might also consider a dual to the cautious enumerations. Namely, in the \emph{petulant enumeration} $\ZFC^\bullet$, we enumerate the axioms of \ZFC, but if we should ever find an explicit proof of contradiction in our favored very strong consistent theory $T$, then in a petulant rage we immediately add a contradiction also to the current enumeration, thereby spoiling $\ZFC^\bullet$. (The general idea is proposed in \cite[theorem~7.6]{Feferman1960:Arithmetization-of-metamathematics-in-a-general-setting}, attributed to Steven Orey.) This enumeration is arguably unnatural, but it is not beyond imagination---we might consider a large cardinal set theorist who truly and sincerely believes in the hierarchy of theories from \ZFC\ up through inaccessible cardinals, measurable cardinals, supercompact cardinals and so on, but is so vested in their success that for an unexpected contradiction to arise higher up would undermine the whole picture in a way that was too upsetting to contemplate. Such is the personality behind the petulant enumeration.

In any case, since we have assumed that the strong theory $T$ is actually consistent, the petulant reaction will never come, and so the petulant enumeration results in exactly the usual theory \ZFC. The difference between $\ZFC$ and $\ZFC^\bullet$ is in this sense merely theoretical and can only be realized in nonstandard models and even then only with nonstandard instances of the axioms. Nevertheless, the petulant theory is much stronger than \ZFC\ in consistency strength, for it is equiconsistent with $T$. The reason is that if $\Con(\ZFC^\bullet)$, then the enumeration must never have reached the petulant stage, and so $\Con(T)$ as well. And conversely, from the consistency of $T$ we know both that $\ZFC$ is consistent and that $\ZFC^\bullet$ agrees with \ZFC. So the petulant theory $\ZFC^\bullet$ is equiconsistent with the much stronger theory $T$.

Let me next consider whether there is a double-jumping large cardinal hypothesis. All the usual large cardinal hypotheses are single-jumping, in the sense that they each have a strong consistency strength, but the negations of the hypotheses do not. And yet, we know by theorem \ref{Theorem.eta-neg-eta-incomparable} that there are set-theoretic axioms with the double-jumping feature, so that both the axiom and its negation have strictly higher consistency strength. Is there a double-jumping large cardinal axiom? Is there a large cardinal axiom $A$, so that both $A$ and $\neg A$ have a consistency strength that is strictly stronger than \ZFC? Consider the following tentative large cardinal assertions:

\begin{theorem}
Amongst the assertions of the form:
 \begin{quote}
  ``If this specific computational process halts, then there is an inaccessible cardinal''
 \end{quote}
there are instances with double-jumping consistent strengths---the sentence and its negation have incomparable consistency strength over \ZFC.

There is computable function $f$ for which the assertions
 \begin{quote}
 ``if $f(n)$ halts, then there is an inaccessible cardinal''
 \end{quote}
have strongly independent incomparable consistency strengths, and each statement is double jumping in consistency stregnth.
\end{theorem}

\begin{proof}
Let $f$ be the universal computable function defined relative to the theory $\ZFC+I+\Con(\ZFC+I)$, and let $\psi_n$ be the statement, ``if $f(n)$ halts, then there is an inaccessible cardinal.'' If $n\neq m$, then there is a model $M$ of $\ZFC+I+\Con(\ZFC+I)$ in which $f(n)$ halts and $f(m)$ does not, provably so by \ZFC\ in the model. Because $M$ thinks $\ZFC+I$ is consistent, it also thinks there is a model $N$ of $\ZFC+I+\neg\Con(\ZFC+I)$. The computation of $f$ on $n$ and $m$ is the same in $M$ and $N$, since it had already found the key decision step at a finite stage in $M$, whose natural numbers are an initial segment of those in $N$. Because $f(n)$ halts in $N$, the statement $\psi_n$ is equivalent to $I$, but $N$ thinks $\neg\Con(\ZFC+I)$, so this is a model of $\neg\Con(\ZFC+\psi_n)$. But because $f(m)$ does not halt, the statement $\psi_m$ is vacuously true, and provably so. So the consistency of $\psi_m$ amounts to $\Con(\ZFC)$, which holds as a consequence of $I$ in $N$. So the statements are pairwise incomparable. By controlling any finite pattern of halting in $f$, we can similarly see that the statements are strongly independent.

To see that the statements $\psi_n$ are all double-jumping, consider the model $N$ as above, which satisfied $\neg\Con(\ZFC+\psi_n)$, but in light of $I$, it satisfies $\Con(\ZFC+\neg I)$ and consequently $\Con(\ZFC+\neg\psi_n)$, since $f(n)$ does halt there. Conversely, if we consider $\psi_m$ in the same model, we've already observed $\Con(\ZFC+\psi_m)$ in $N$, and $\neg\Con(\ZFC+\neg\psi_m)$ holds there, since $f(m)$ provably does not halt. If we make this same analysis with $n$ instead of $m$, we see that $\psi_n$ and $\neg\psi_n$ are incomparable in consistency strength, and so $\psi_n$ is double-jumping.
\end{proof}

We could have used essentially any strong statement in the conclusion, instead of asserting the existence of inaccessible cardinals; for example, we could have said that if the computation halted, then there is a strong cardinal or a proper class of Woodin cardinals or what have you. In addition, essentially the same argument works with other forms of existential hypotheses. For example, statements of the form
\begin{quote}
 ``If this specific diophantine equation $p(\vec x)=0$ has a solution in the integers, then there is a supercompact cardinal''
\end{quote}
or the form
\begin{quote}
 ``If this specific set of polygonal tiles admits no tiling of the plane, then there is an almost huge cardinal''
\end{quote}
will also respectively exhibit instances of double-jumping and strongly independent incomparability in consistency strength. Are these large cardinal axioms? They do seem to make at least provisional large cardinal assertions, and in this sense they might be regarded as large cardinal axioms. Such a provisional nature about consistency, I claim, is a necessary feature of any double-jumping hypothesis, because if a statement provably implies $\Con(ZFC)$, then it cannot be double-jumping, since the negation of the statement would follow from $\neg\Con(\ZFC)$, which has no consistency strength at all. If a statement is to be double-jumping, therefore, then it must be provisional in this way about consistency.

\section{Nonlinearity in the hierarchy of transitive-model existence}

Let us consider an analogue of the hierarchy of consistency strength, but where one requires the existence of transitive models of the theory, rather than mere consistency. In the hierarchy of \emph{transitive-model existence} strength, we say that theories are related $S\leq T$ if we can prove in the base theory \ZFC\ that the existence of a transitive model of $T$ implies the existence of a transitive model of $S$.

It turns out that many of the nonlinearity results carry over almost unchanged to this revised hierarchy.\goodbreak

\begin{theorem}
 There is a computable function $f$ for which the assertions
 \begin{quote}
   ``there are $f(n)$ inaccessible cardinals''
 \end{quote}
 are incomparable and strongly independent in the hierarchy of transitive-model existence.
\end{theorem}

\begin{proof}
We use the same function $f$ as in theorem \ref{Theorem.Incomparability-f(n)-inacc}. If $n\neq m$, then as before there is a model $M$ with infinitely many inaccessible cardinals in which $f(n)<f(m)$, and both are defined. Inside this model $M$, let $N$ be any $\in$-minimal transitive model of \ZFC\ that has a transitive model with $f(n)$ many inaccessible cardinals. It follows by minimality that this model can have no transitive model with $f(n)+1$ many inaccessible cardinals, and so $N$ is a model that thinks ``there are $f(n)$ inaccessible cardinals'' has a transitive model, but ``there are $f(m)$ inaccessible cardinals'' does not. Therefore the transitive model existence statement for $n$ does not imply the statement for $m$, and so these statements are all pairwise incomparable in that hierarchy. By considering finite patterns with the universal computable function, we similarly achieve that these statements are strongly independent.
\end{proof}

\section{Nonlinearity in the largest-number contest}

Allow me briefly to digress with a discussion of how some of these issues play out with the \emph{largest-number} contest, in which contestants compete to describe the largest number subject to certain constraints of space and language; they write their submissions on a standard index card, with rules specifying which characters are allowed and how many. Perhaps one contestant fills their card naively with {\tt 9}s in decimal notation, while another describes a much larger number with factorials {\tt 9!!!!!!!!} or with iterated exponentials {\tt 2\^{}2\^{}2\^{}2\^{}2\^{}2}. On even a moderately sized card, we can describe some truly large numbers.

Although I did once supervise and judge an actual instance of the game with competitors from the audience of a large public lecture I gave in Shanghai \cite{Hamkins.blog2013:Largest-number-contest}, nevertheless the largest-number contest is more often played in the imagination---it is more thought experiment than actual game. The reason is that the game leads quickly to difficult metamathematical matters, which begin to arise when one allows number descriptions going beyond mere primitive recursive terms. For starters, if one allows a free-form descriptive language, then one will immediately engage with Berry's paradox in submissions such as, ``the largest number that can be described on $3\times 5$ index card, plus $1$,'' a description that itself fits easily on a $3\times 5$ index card, but by doing so serves up a paradox.

One might hope to avoid paradox by restricting the language, allowing only precise definitions in a formal language, say, such as those of the form ``the smallest number $n$ such that $\varphi(n)$,'' where $\varphi$ is a formula in the first-order language of arithmetic; or perhaps one allows submissions of the form, ``the output of this specific Turing machine computation.'' These kinds of submission might initially seem tame, but how are we to know whether indeed there is such a number $n$ for which $\varphi(n)$ or whether the submitted computational procedure will ever halt? Perhaps we should insist that submissions be accompanied by a proof that the submission does indeed define a number. For example, such a proof is easy to provide for definitions of the form, ``the smallest number $n$ such that $\varphi(n)$, if any, otherwise $17$.''

And yet, the judge will have difficulty to decide who has won. Even in the case where the language is restricted to primitive recursive terms of the form:
  \begin{quote}
  {\tt googol plex bang bang stack}

  or

  \noindent{\tt googol stack bang plex plex},\footnote{Following pop-math practice, a \emph{googol} is $10^{100}$; the expression $x$ {\tt bang} refers to the factorial $x!$; the expression $x$ {\tt plex} means $10^x$; and $x$ {\tt stack} means the $x$-iterated exponential $10^{10^{\cdot^{\cdot^{10}}}}$.}
  \end{quote}
and so on, then it is an open question whether in the general case there is any feasible algorithm to determine the larger number; see \cite{Hamkins.MSE72646:Help-me-put-these-enormous-numbers-in-order}.

There is no computable procedure at all to determine whether submissions of the form ``the running time of this computation'' are legitimate submissions, since the judge would have to solve these instances of the halting problem. The same holds for determining the winner for submissions of the form ``the running time of this computation, if it halts, otherwise $17$,'' for if another player submits $18$, then the judge will have to know whether that computation halts to adjudicate the winner.

Apart from computable undecidability, however, there is the more subtle issue of logical undecidability. We might suppose that the judge is presiding over the contest in the context of a fixed official background theory $T$, a consistent computably axiomatizable theory that she relies upon when adjudicating the comparative sizes of the number descriptions. Perhaps this theory is very strong---it might be \ZFC\ plus an aspirational large cardinal hypothesis, amongst the strongest theories thought to be consistent. The problem, of course, is that no consistent computably axiomatizable theory will settle all the comparisons that might be put to it; there will be concrete number descriptions for which the theory does not provably settle their comparative size, and so the judge will not be able to declare the winner. One player might submit, ``the output of this computation'' and another ``the output of that computation,'' but just as the theory does not settle the values of the universal computation, it will not be able to settle who has won.

A determined game-player might argue that we want the judge not to use a theory at all, but rather to use \emph{actual arithmetic truth}, truth in the standard model $\N$. That is, we want to declare as the winner the number description that actually gives rise to the largest number, when this description is interpreted in the standard model $\N$. In \ZFC\ we can prove that there is a definite truth predicate on the standard model $\N$, which can be used to adjudicate these comparisons.

But this proposal is less helpful than one might expect. What good is using a complete truth predicate, after all, if we don't know what the predicate's truth judgements are? Although \ZFC\ and our other foundational theories prove that there is a definite arithmetic truth predicate, these theories do not tell us fully which arithmetic statements are true. So this proposal does not seem to enable the judge to determine the winner. Faced with two number descriptions, even if the judge might know that one of them describes the actually larger number, she may have no way to find out which one.

\enlargethispage{20pt}
We can summarize the basic facts in the following theorem.

\begin{theorem}
 For any given consistent theory $T$, there are entries for the largest-number contest of the form, ``the first $n$ with this concrete property $\varphi(n)$,'' for which it is undecidable in $T$ which player has won, and furthermore, for which the respective assertions that a given player has won have incomparable consistency strengths.
\end{theorem}\goodbreak

\section{Illusory linearity and the confirmation bias argument}

Let me discuss a few reasons for thinking that we might simply be mistaken about the linearity phenomenon. Perhaps it is illusory?

One observation tending to undermine the linearity evidence is simply the fact that many common large cardinal notions were constructed specifically to strengthen previous notions, often in well-ordered hierarchies. The progressions from inaccessible cardinals to hyperinaccessible and the hyperinaccessibility hierarchy and from Mahlo through the hyper-Mahloness hierarchy and ultimately to greatly Mahlo occurred exactly like this in the early days of set theory. Similarly with the progression from measurable cardinals to hypermeasurability and through the strongness hierarchy to strong cardinals. Solovay generalized the strongly compact to the supercompact cardinals with an embedding characterization that led naturally to many further strengthenings, among them the almost huge, huge, and superhuge cardinals. These conceptions altogether firmly established the paradigm of large cardinals as critical points of embeddings $j:V\to M$, which constitute the large-scale bones of the large cardinal hierarchy. One naturally strengthens these axioms simply by insisting upon progressively stronger closure properties on $M$. These strengthened notions are often therefore linearly related in a way that is completely unsurprising---they form a linear hierarchy precisely because that is how we created them.

Mirna \Dzamonja\ describes the situation like this:
\begin{quote}
The linearity pretty obviously seems to be just a consequence of definitions mostly being variants of each other. Increase the target, increase the closure\ldots clearly once we get to be more inventive we shall have no linearity. \cite{Dzamonja.Twitter2021:Linearity-in-large-cardinals}
\end{quote}

This linear creation process occurs even at the very bottom of the hierarchy of consistency strength, with the tower of iterated consistency assertions, at each transfinite stage adding the consistency statement of the theory that has come before. Of course this creates a well-founded linear tower of consistency strength.

These observations show that huge parts of the consistency strength hierarchy exhibit linearity only for superficial, unsurprising reasons, which therefore cannot count as evidence of a broader or more fundamental linearity phenomenon. The unexpected instances of linearity are simply many fewer than one might have expected. Do those instances suffice to establish linearity as a genuine phenomenon?

The defenders of linearity might reply to this criticism by pointing out that it is not only the large cardinal notions themselves that are linearly ordered by consistency strength, but all the other statements in mathematics that have been proved equiconsistent with various large cardinal notions. Work of Solovay and Shelah shows that the impossibility of removing the axiom of choice from Vitali's construction of a nonmeasurable set is equiconsistent with the existence of an inaccessible cardinal; dramatic work on determinacy establishes equiconsistency connections with Woodin cardinals; and so on in many other cases. From this point of view, it doesn't much matter that the large cardinal notions themselves are often linearly ordered in a trivial manner, if all these other statements from mathematics are proved to line up with that hierarchy---it would still establish a genuine linearity phenomenon.

But does this counterargument overplay its hand? We don't actually have \emph{so} many instances of equiconsistency between large cardinals and mathematical principles arising outside logic and set theory. There are some very prominent cases, but are these truly sufficient to make the case for a genuine widespread linearity phenomenon? I don't believe it is enough.

A second observation undermining the linearity evidence is the simple fact that in a number of irritating instances, we don't yet actually know the natural hierarchy to be linear. For example, we don't yet know well exactly how the strongly compact cardinals fit into the hierarchy, even though this was one of the earliest large cardinal notions, arising from compactness properties of infinitary logic. The defenders of linearity often shrug at this case, brushing off concerns about it, but I find that strange and frankly unsatisfying, since this is a core large cardinal notion that directly undermines or at least challenges the claim of a sweeping linearity phenomenon for large cardinals. For example, we just don't know how to compare the consistency strength of one supercompact cardinal with two strongly compact cardinals; we don't know the strength of an indestructible weakly compact cardinal; we don't know exactly how the proper forcing axiom fits into the hierarchy; and similarly in many other cases. As far as we know, there could be abundant nonlinearity in consistency strengths surrounding each of these and many other commonly studied axioms.

Let me next explain a subtler argument suggesting that we may be mistaken about linearity, namely, what I call the \emph{confirmation-bias} argument. When engaging with the independence phenomenon in set theory and establishing relative consistency strengths, we generally begin with a model of one theory $T$ and construct from it a model of another theory $S$. From any model of \ZF, for example, we construct models of $\ZFC+\CH$ or $\ZFC+\neg\CH$ or Martin's axiom or what have you; from a model of \ZFC\ with sufficient large cardinals, we construct models of $\ZF+\DC$ with various determinacy axioms. When we begin with a model of $T$ and construct a model of $S$, then we will have established $S\leq T$ in consistency strength, and when there is also a converse construction, then we will also know $T\leq S$ and consequently that the theories are equiconsistent, $S\equiv T$. When the model of $S$ that we construct is a set model inside the model of $T$, then by the sufficiency condition I mentioned in section \ref{Section.Formal-nonlinearity}, we deduce the strict relation $S<T$.

In nearly all these arguments, the tools we use are forcing and transitive inner models (including set-sized models), often in sophisticated elaborate combination. Set theorists have become expert in using these tools to explore the vast range of set-theoretic possibility.

The confirmation-bias observation is that with this process and these tools by themselves, we shall never establish an instance of incomparability in consistency strength. The reason is that both methods preserve arithmetic truth---we cannot change the arithmetic of a model by forcing or by going to a transitive set or transitive class inner model. But in order to establish nonlinearity in consistency strength, it is necessary that we change the arithmetic truths of the models. To show that theories $S$ and $T$ are incomparable, after all, we need to provide a model of \ZFC\ with $\Con(S)$ but not $\Con(T)$ and another with $\Con(T)$ but not $\Con(S)$, and this will require that these models have not only different natural number structures $\N$, but end-extension incomparable such natural number structures---and so they will also be interpretatively incomparable, for in any interpretation the inconsistency of the other model would still be present and we wouldn't achieve the incomparability situation.

The final conclusion of the confirmation-bias argument is that we shouldn't be surprised to observe only linearity, if our tools are incapable of observing nonlinearity. This is precisely what it means to suffer a confirmation-bias error.

\section{Is self-reference disqualifying for naturality?}

Many people object to the naturality of the arithmetic sentences mentioned in section \ref{Section.Formal-nonlinearity} on the grounds that these sentences are self-referential. Perhaps it is thought that self-reference is a strange and unexpected feature in mathematics, and therefore it may perhaps be disqualifying for a self-referential assertion to be considered ``natural.'' I should like to push back against this view.

I find diagonalization and self-reference to be at the very core of mathematical logic and set theory. Set theory as a subject essentially begins with Cantor's diagonal argument, establishing uncountability in an instance of diagonal self-application. Russell's refutation of the general comprehension axiom is explicitly self-referential, building the class of all non-self-membered sets. We don't take general comprehension to be fine otherwise, for ``natural'' assertions. These arguments are surely amongst the founding central ideas of the subject, and the diagonalization idea is woven deeply throughout it. Furthermore, these diagonalizations are fundamentally the same as used to prove the fixed-point lemmas that lead to the \Godel\ and Rosser sentences. What can be the coherent philosophy of ``natural'' that counts the constructions of Cantor and Russell as natural, but not the fundamentally similar construction of the \Godel\ and Rosser sentences?

To my way thinking, the large cardinal axioms themselves engage in a kind of self-reference. To assert that there is an elementary embedding $j:V\to M$ of a certain kind is to assert that there is transformation of objects from our current world $V$ to a new world $M$ associating to every object $x$ in our world a duplicate $j(x)$ in the new world, with all the same properties in that world as $x$ had in the old world. So the axiom at bottom posits a system of duplicates $j(x)$, whose properties are stated by (self-)reference back to $x$. In this light, nearly every large cardinal axiom partakes of self-reference.

Or consider the proper forcing axiom \PFA, a strong generalization of Martin's axiom \MA\ from the case of c.c.c. forcing to proper forcing. This axiom is commonly included amongst the ``natural'' set-theoretic hypotheses, with most bets placing it as equiconsistent with (or very near to) the existence of a supercompact cardinal in consistency strength. The axiom makes the Martin's axiom claim about the existence of a filter meeting any given $\aleph_1$ many dense sets for a forcing notion $\P$, provided that it is proper. But to my way of thinking, the property of a forcing notion $\P$ being proper is at least as self-referential as the Rosser sentence or the other simple arithmetic sentences discussed earlier in this article (and an order of magnitude or two more difficult). Namely, $\P$ is proper essentially when $\P$ densely often contains conditions that are themselves generic or ensure genericity for the forcing with the version of $\P$ itself over suitable countable elementary substructures of sufficient fragments of the universe; put simply, $\P$ is proper when it consists of conditions that are generic for the forcing $\P$ itself, in countable simulacra over various countable universes. Is this any less self-referential than a sentence making an assertion about its own \Godel\ code? Why should the comparatively simple Rosser sentence be somehow beyond the pale, while properness counts as natural?

Diagonalization and self-reference are pervasive in logic and set theory, and to regard those features as automatically unnatural would have us declaring the entire subject unnatural.

\section{Naturality}

What does it mean to have a ``natural'' example in mathematics? Many mathematicians seem to adopt a know-it-when-you-see-it attitude to naturality, without giving a formal account. Does it matter that what counts as natural sometimes changes over time?

There is the connotation that natural examples are those occurring in practice---Koellner \cite{Koellner2011:SEP-independence-large-cardinals} refers to examples that ``arise in nature'' and Steel \cite{Steel2014:Godels-program} says that natural set-theoretic hypotheses are those ``considered by set theorists, because they had some set-theoretic idea behind them.''

It would be the naturalist fallacy, of course, to imbue that conception of naturality with all the positive connotations one usually finds for this word in mathematics. Is every call for natural examples in mathematics commiting the naturalist fallacy? Is the subject of mathematics rife with the naturalist fallacy?

I don't believe so, because despite those remarks of Koellner and Steel, naturality as it is used in mathematics is usually not a reference to actual practice. After all, one commonly hears naturality-based objections to examples offered in actual practice; mathematicians often criticize even a well-established argument or example, requesting a more natural one. Consider the subject of computability theory, for example, which has thousands of constructions using complex priority arguments to establish certain fundamental features in the hierarchy of Turing degrees. These are computability-theoretic constructions, introduced by computability theorists for computability-theoretic purposes, but the resulting degrees are nevertheless sometimes criticized as unnatural. One hears requests for ``natural'' Turing degrees with certain features---for example, is there a natural solution of Post's problem?

(Another instance, which I hesitate to mention: I am a set theorist who has in this very paper introduced theories and axioms for a set-theoretic purpose, defending them as natural and set-theoretic, while proving that they constitute instances of nonlinearity and illfoundedness. But are we to take this by itself as refuting Steel's conjecture, mentioned on page \pageref{Steel-conjecture}, I doubt that I shall succeed so easily.)

A different conception of the arising-in-practice connotation would be that natural notions are those that might easily arise in an unrelated subject or practice. On this view, examples in logic would count as natural if they might arise in graph theory, algebra or topology. An actual graph-theoretic decision problem would be seen as natural in computability theory. But this view, taken strictly, would seem to rule out many of the large cardinal hypotheses that Steel and Koellner want to see as natural, since other subjects generally give very little consideration to large cardinals at all. We need a different notion of the ``natural.''

There is a connotation that natural examples must be examples that have or at least could arise independently of whatever immediate application is currently being made of them. This would explain why many computability theoretic constructions are often seen as unnatural, as those constructions are often aimed solely at providing instances of the specific matter at hand. Furthermore, this property would also be true whenever an example was not overly detailed or technical, since one can easily imagine that simple examples could arise independently of any particular motivation. Certainly in many instances, the natural examples seem to be those that are easily described or presented.

For such a view, however, one might recall Shelah's response to criticisms of his arguments as being overly ``technical'':
\begin{quote}
The term technical is a red flag for me, as it is many times used not for the routine business of implementing
ideas but for the parts, ideas and all, which are just hard to understand and many
times contain the main novelties. \cite[Axis~A]{Shelah1993[Sh:E16]:The-future-of-set-theory}
\end{quote}
One might respond similarly to claims of unnaturality. Would we want to hold that all difficult examples are unnatural? Earlier I had mentioned the proper forcing axiom \PFA, which is surely difficult yet often taken by set theorists as natural. And many large cardinal notions, including Woodin cardinals, remarkable cardinals, and many others, are commonly regarded as difficult.

Let me mention that the actual-practice conception of naturality, regrettably, can too easily be construed narrowly in a way that harms mathematical advance. Namely, in my experience concerns about naturality are sometimes raised in effect simply to reject unfamiliar ideas or constructions. In a few instances on MathOverflow, a mathematician not from logic asked a mathematical question that I was able to answer easily using ideas from set theory and logic; perhaps my argument used transfinite recursion or some other standard method, completely natural to my way of thinking. But the solution was rejected as unnatural---the person was just not knowledgable of the subject. That would have been fine, if they had been open to learning, I would have explained more, but to reject the approach as ``unnatural'' and be done with it is too much---I find the criticism empty. What a pity it would be for our conception of the natural to lead us all into scattered knowledge silos.

A similar objectionable use of ``natural'' arises when it is simply a cover for sloppiness in one's mathematical conceptions. Perhaps someone asks, perhaps again on MathOverflow, whether every instance of $X$ has feature $Y$, and when a counterexample is produced, the objection is made that it is ``unnatural.'' But of course, naturality might have little to do with it---often the question had simply been ill formulated; the person had asked about $X$ but greater reflection might have lead them to formulate additional unstated properties $X'$ or $X''$. In light of such experiences, I tend toward suspicion of any free-and-easy use of ``natural'' in mathematics---to say that a theory or example is ``natural'' or ``unnatural'' is too often simply empty, or at best a lazy stand-in for other unstated properties that the speaker has not yet formulated clearly.

%

Carolin Antos-Kuby suggested at my talk on this paper in her seminar that set theorists of the universist philosophy, the view that there is a unique intended universe of set theory, might find it appealing to say that the ``natural'' statements and theories are simply those that are true in the one true universe. This could explain why theories such as $\ZFC+\neg\Con(\ZFC)$ and $\ZFC+$``there is no Mahlo cardinal'' are found unnatural by the large cardinal universists. On this view every statement or its negation will be natural---and the most natural set theories will ultimately converge to the one true set theory. My reply to this suggestion is that it will not achieve the linearity claim for the hierarchy of consistency strength. Indeed, it will contradict linearity, since all my main examples of incomparability and ill-foundedness, the cautious theories and so on, are all true set theories on the usual large cardinal universist view. The cautious theories are subtle weakenings of the large cardinal theories, but weakened in a naturally cautious manner that causes them to have incomparable consistency strength. So they would count as ``natural'' on this proposal, yet still exhibit nonlinearity.

Consider a related naturality objection that one might imagine against the incomparability sentences of sections \ref{Section.Natural-instances}, \ref{Section.Natural-instances-arithmetic}, and \ref{Section.Cautious-theories}. Namely, the differences between the sentences on offer are revealed only in $\omega$-nonstandard models, which is an unnatural case. The differing natures of the computable function $f$ used in theorem \ref{Theorem.Incomparability-f(n)-inacc}, for example, can be exhibited only in $\omega$-nonstandard models; in standard models, the sentences will always agree.

My rebuttal to this objection is that consistency strength is \emph{inherently} about $\omega$-nonstandard models. If we only ever considered $\omega$-standard models, they would all always agree about the consistency of any given theory whatsoever, and there would be no hierarchy to analyze. For one theory to be weaker than another in consistency strength, $S<T$, means exactly that there is a model of $\Con(S)+\neg\Con(T)$, and this model must be $\omega$-nonstandard if $T$ is consistent. Therefore any substantive treatment of the consistency strength hierarchy must ultimately be concerned with nonstandard models.

Let me next raise the distinction between a natural \emph{kind} of problem and a natural problem instance. The polygonal tiling problem, for example, is surely a natural kind of problem, but does this mean that every individual tiling problem is natural? If not, this might speak against the success of my examples in corollary \ref{Corollary.Tiling-problems} as natural instances of nonlinearity. That is, tiling problems are a natural class of problem, and I found specific tiling problems with incomparable consistency strength, but perhaps those specific problems do not count as natural considered on their own. Similarly, perhaps the incomparability and ill-foundedness examples provided by theorems \ref{Theorem.Incomparability-f(n)-inacc}, \ref{Theorem.Incomparability-running-time}, \ref{Theorem.Incomparability-e-halts} and so on are merely instances of natural kinds of problems, rather than actually natural instances themselves.

The respective undecidability of the tiling problem, the diophantine problem, the halting problem, the Entsheidungsproblem, and many others holds profound philosophical significance in identifying fundamental limitations on our ability to achieve mathematical knowledge mechanistically. The undecidability results show that we can have no uniform computable procedure to solve instances of these extremely natural kinds of problems. Yet, the undecidability results themselves are all proved ultimately by means of diagonal arguments that make use in each case of weird self-referential instances of those problems.  Turing's argument for the undecidability of the halting problem, for example, invokes a program that asks about its own behavior when applied to itself as input (which is something we basically never do in practice). And this weird example is in effect copied into all the others, when one proves undecidability through a reduction of the halting problem.

But do we say that the halting problem is decidable for ``natural'' instances? No, even though for the programs arising in practice, those written for a clear computational purpose, we can indeed generally determine whether it will halt or not. (Indeed, Miasnikov and I proved \cite{HamkinsMiasnikov2006:HaltingProblemDecidable} that there is a linear-time algorithm that correctly decides almost every instance of the halting problem, with respect to the asymptotic density measure---as the number of states $n$ increases, the proportion of all $n$-state programs handled by our algorithm approaches 100\%.) But in computability theory we do not generally say that the halting problem is essentially decidable; we do not highlight a ``natural'' halting-problem-decidability phenomenon, according to which we can decide halting in the natural cases that matter. Why do the set theorists claim that consistency strength is linear in the natural cases that matter?

Steel \cite[slide 26]{Steel2013:Godels-program-stanford-slides} advances the idea that the consistency strength order of natural theories aligns with containment of their arithmetic consequences. Specifically, the claim is that for natural theories for which $S<T$ in consistency strength, then any arithmetic consequence of $S$ will also be provable in $T$. I pointed out in \cite{MO236436.Hamkins:Does-consistency-strength-hierarchy-coincide-with-arithmetic-consequence?}, however, that Steel's principle will require us to rule out the naturality of many theories that we might otherwise have been inclined to accept as natural. Specifically, if $T$ is a natural large cardinal theory, then it seems to me that many set theorists would readily accept $\Con(T)$ as a natural assertion. Indeed, for many natural theories this is an assertion that has appeared as an explicit hypothesis in perhaps hundreds of published theorems in the set-theoretic research literature, and in this sense fulfills the actual-practice conception of naturality.

But to regard $\ZFC+\Con(T)$ as natural would in most cases contradict Steel's arithmetic containment principle. To see this, assume that $T$ proves the existence of an inaccessible cardinal $\kappa$. It follows that $T$ will prove all instances of the reflection scheme asserting, ``if arithmetic statement $\phi$ is provable in \ZFC, then it is true.'' The reason is that $T$ will prove that $\phi$ holds $V_\kappa$, and since $\phi$ is arithmetic, it will be absolute to the full set-theoretic universe $V$. Meanwhile, the consistency-wise stronger theory $\ZFC+\Con(T)$ cannot prove all instances of this scheme, because this theory is consistent with the assertion $\neg\Con(\ZFC+\Con(T))$, and any model of this combined theory will think that $\neg\Con(T)$ is provable in $\ZFC$ but not true. Therefore, we have a natural theory $T$ that is strictly weaker in consistency strength than another theory $\ZFC+\Con(T)$, yet the weaker theory proves some arithmetic statements that are not provable in the stronger theory. This would contradict the arithmetic containment principle, unless we regard $\ZFC+\Con(T)$ as unnatural. Therefore, if one is committed to the idea that consistency strength increases align with containment of arithmetic consequences for natural theories, then we cannot allow $\ZFC+\Con(T)$ as a natural theory even when $T$ is.

Steel will reply, of course, that indeed $\ZFC+\Con(T)$ is not a natural theory---this theory instantiates exactly what he describes as the \emph{instrumentalist dodge} in set theory, described in \cite[p.~423]{FefermanFriedmanMaddySteel2000:Does-mathematics-need-new-axioms}. Namely, we don't want to assume merely that the large cardinals are consistent, but rather that they are actually true. What I have just argued here is that the arithmetic-containment principle for consistency strength requires this stance. In particular, according to this view the theory $\ZFC+\Con(\ZFC+\exists\text{ inaccessible})$ is not natural.

\section{Analogy between set theory and computability theory}

Let me draw an analogy between set theory with its study of the hierarchy of consistency strength and computability theory with its study of the Turing degrees, a rich hierarchy of complexity that is surely as deep and complicated as the hierarchy of consistency strength, and also as philosophically significant---one can view the Turing degrees as the possible countable amounts of information. Russell Miller has described what he calls the ``build it and they will come'' philosophy in computability theory, according to which if one wants to exhibit a certain feature in the hierarchy of Turing degrees, then you simply have to get down to business and make it happen with a particular construction built for the purpose. Computability theorists seem quite commonly to embrace the chaotic scrappiness of the hierarchy of Turing degrees. I wonder whether such an attitude towards the hierarchy of consistency strength in set theory would lead us to discover fascinating new phenomenon in the degrees of consistency strength.

Meanwhile, computability theorists also point to their own natural linearity phenomenon, namely, the ``naturally arising'' Turing degrees invariably arise in a linear, well-ordered part of the hierarchy of Turing degrees. The commonly arising definable sets of natural numbers, the ones we might be independently interested in, tend to have their Turing degrees landing precisely on one of the low-level iterated jumps:
   $$0<0'<0''<0'''<\cdots<0^{(\omega)}<\cdots$$
In this sense the ``natural'' Turing degrees are well-ordered, and researchers seek a deeper explanation. Joseph Miller points to the research efforts around the Sacks question, asking for a degree-invariant solution of the generalized Post's problem, and the Martin conjecture, seeking to establish the Turing jump and its iterates as canonical for definable degree-invariant Borel actions, as part of the program to provide a deeper explanation of the linearity phenomenon in the Turing degrees. Antonio Montalb\'an \cite{Montalban2019:Martin's-conjecture-a-classification-of-the-naturally-occuring-Turing-degrees} explains the importance of Martin's conjecture like this:
\begin{quote}
The [linear] hierarchy we were looking for seems to exist, but $\mathcal{D}$ [the Turing degrees] seems too chaotic to help us find it. The contrast between the general behavior in $\mathcal{D}$ and the behavior of the naturally occurring objects is so stark that there must be a deep reason behind it. We need to dig deeper.
\end{quote}

Those conjectures in effect seek to replace naturality talk by identifying exactly the properties that are sought: degree-invariance and restrictions to Borel actions in place of arbitrary actions. One needn't refer any longer to the ``natural'' Turing degrees to engage with them, but can rather use these more specific ideas. The computability theorists have thus filled in what ``natural'' means here.

\section{A challenge for defenders of natural linearity}

Let me close this article on a positive note with a challenge to the defenders of the natural linearity phenomenon. In light of the abundant counterexamples establishing pervasive nonlinearity and ill-foundedness in the hierarchy, I propose that we should abandon the empty naturality talk and instead get down to the work of identifying the attractive features we had sought in our notion of the natural. Can we give legs to a reified naturality notion that is sufficient to establish linearity in the consistency-strength hierarchy? For example, what is the set-theoretic analogue of the Martin conjecture for consistency strength? Might we ultimately hope to identify a broad class of assertions with welcome, attractive features---standing in for the so-called ``natural'' assertions---which provably align into a well ordered hierarchy of consistency strength? That would be how to do it.

Here is one small step in this direction, which I provided at \cite{Hamkins.MO384050:Uniform-incomparable-consistency-strengths} in response to a question of Dmytro Taranovsky, who told me he was inspired to ask it because of this article.

\begin{theorem}
 There is no monotone analogue of the independent Rosser sentence construction. That is, there is no assignment $\tau\mapsto \rho_\tau$ of sentences $\tau$ in the language of arithmetic to sentences $\rho_\tau$ with the properties:
 \begin{enumerate}
   \item (Independence) If $\PA+\tau$ is consistent, then so are $\PA+\tau+\rho_\tau$ and $\PA+\tau+\neg\rho_\tau$.
   \item (Extensionality) If $\PA\proves\tau\iff\sigma$, then $\PA\proves\rho_\tau\iff\rho_\sigma$.
   \item (Monotonicity) If $\PA\proves\tau\to\sigma$, then $\PA\proves\rho_\tau\to\rho_\sigma$.
 \end{enumerate}
\end{theorem}

The extensionality principle (also known as uniformity) expresses that the sentence $\rho_\tau$ does not depend on intensional aspects of the manner in which $\tau$ is asserted or the proof system, but is well-defined up to provable equivalence. The three properties are redundant, of course, since monotonicity is a strengthening of extensionality.

\begin{proof}
Since $\PA$ is consistent, it follows that $1=1$ is consistent with \PA, and so $\rho_{1=1}$ is independent of \PA. In particular, $\neg\rho_{1=1}$ is consistent with \PA, and so $\PA+\neg\rho_{1=1}+\rho_{\neg\rho_{1=1}}$ is consistent. But \PA\ trivially proves that $\neg\rho_{1=1}\to 1=1$, and so by monotonicity it follows that \PA\ proves $\rho_{\neg\rho_{1=1}}\to\rho_{1=1}$, which contradicts the earlier stated consistency.
\end{proof}

In particular, the Rosser sentence itself does not obey monotonicity. Can one weaken monotonicity to mere extensionality? Shavrukov and Visser \cite{ShavrukovVisser2014:Uniform-density-in-Lindenbaum-algebras} provide a uniform computable construction producing a sentence $F(A,B)$ that is strictly between $A$ and $B$ in the Lindenbaum algebra over \PA, whenever $A$ is strictly below $B$ in that algebra, and which furthermore is extensional, in the sense that if $A$ and $B$ are replaced with equivalent $A'$ and $B'$, then $F(A,B)$ is equivalent to $F(A',B')$. They use this construction to provide extensional Rosser constructions. The resulting sentences, however, are not $\Pi^0_1$. This seems to leave the following question open:

\begin{question}
 Is there a $\Pi^0_1$ formula $\rho(x)$ with the following properties?
 \begin{enumerate}
   \item (Independence) If $\PA+\tau$ is consistent, then so are $\PA+\tau+\rho(\gcode{\tau})$ and $\PA+\tau+\neg\rho(\gcode{\tau})$.
   \item (Extensionality) If $\PA\proves\tau\iff\sigma$, then $\PA\proves\rho(\gcode{\tau})\iff\rho(\gcode{\sigma})$.
 \end{enumerate}
\end{question}

A positive answer would provide a degree of nonlinear naturality, thereby undermining the linear naturality hypothesis, whereas a negative answer would support the hypothesis, in showing that independence must be sensitive to intensionality.

My challenge is not just this one question, but the challenge of producing many more such questions, aimed at giving legs and mathematical substance to our conception of what would count as ``natural'' instances of independence and incomparability in the hierarchy of consistency strength.

Recent work by James Walsh and others are deeply engaged with proof-theoretic aspects of the natural linearity phenomenon, including proof-theoretic analogues of Martin’s Conjecture \cite{MontalbanWalsh2019:On-the-inevitability-of-the-consistency-operator, Walsh2020:A-note-on-the-consistency-operator, Walsh2022:Evitable-iterates-of-the-consistency-operator}, and other work on what they refer to as the well-ordering phenomenon for natural theories \cite{PakhomovWalsh2021:Reflection-ranks-and-ordinal-analysis, PakhomovWalsh2021:Reflection-ranks-via-infinitary-derivations, Walsh2022:A-robust-proof-theoretic-well-ordering}. In a sense, Walsh's project takes the well-ordering phenomenon as a given starting point, seeking then to answer the question: what is the meaning of ``natural'' to make it true that natural theories are well-ordered by consistency or interpretability strength? According to Walsh, ``the emerging picture is that natural theories are proof-theoretically equivalent to iterated reflection principles'' \cite{Walsh2021:On-the-hierarchy-of-natural-theories}.

Much of the analysis concerns the kind of uniformity and invariance requirements that appear in Martin's conjecture, but for theories and interpretation instead of Turing degrees and relative computability. While those requirements may help to establish the well-order phenomenon, ultimately I would find it disputable whether such requirements are actually part of the notion of ``natural'' as it is commonly used in large cardinal set theory. The typical case there has one-off theories, such as \ZFC+``there is a supercompact cardinal,'' which are regarded as natural, but without any intention to realize them as an instance of uniform procedure of extending arbitrarily given theories. Does the uniformity analysis provide an account of linearity for this kind of use of ``natural''? The main philosophical counterpoint here is that one cannot convincingly establish the natural linearity phenomenon, after all, by presuming that it is true and then (re)defining a notion of ``natural'' so as to give rise to it. But far be it from me to object that some notions of ``natural'' may be unnatural. On the contrary, my proposal is that we should view naturality talk as a stand-in for other more precise notions, which our favored natural theories exhibit.

\printbibliography


\end{document}